\documentclass[12pt]{amsart}
\usepackage{cases}
\usepackage{mathrsfs}
\usepackage{color}
\usepackage{latexsym,amssymb}
\usepackage{a4wide}
\usepackage{amsthm}
\usepackage{amsmath}
\usepackage{graphicx}
\input xy
  \xyoption{all}
\usepackage{verbatim}
\usepackage[lite,abbrev,msc-links]{amsrefs}

\numberwithin{equation}{section}
\begin{document}
\theoremstyle{plain}
\newtheorem{thm}{Theorem}[section]
\newtheorem{lem}[thm]{Lemma}
\newtheorem{cor}[thm]{Corollary}
\newtheorem{cor*}[thm]{Corollary*}
\newtheorem{prop}[thm]{Proposition}
\newtheorem{prop*}[thm]{Proposition*}
\newtheorem{conj}[thm]{Conjecture}
%%%%%%%%%%%%%%%%%%%% Text roman %%%%%%%%%%%%%%%%%%%%%%%%%%%%%
\theoremstyle{definition}
\newtheorem{construction}{Construction}
\newtheorem{notations}[thm]{Notations}
\newtheorem{question}[thm]{Question}
\newtheorem{prob}[thm]{Problem}
\newtheorem{rmk}[thm]{Remark}
\newtheorem{remarks}[thm]{Remarks}
\newtheorem{defn}[thm]{Definition}
\newtheorem{claim}[thm]{Claim}
\newtheorem{assumption}[thm]{Assumption}
\newtheorem{assumptions}[thm]{Assumptions}
\newtheorem{properties}[thm]{Properties}
\newtheorem{exmp}[thm]{Example}
\newtheorem{comments}[thm]{Comments}
\newtheorem{blank}[thm]{}
\newtheorem{observation}[thm]{Observation}
\newtheorem{defn-thm}[thm]{Definition-Theorem}
\newtheorem*{Setting}{Setting}

\newcommand{\sA}{\mathscr{A}}
\newcommand{\sB}{\mathscr{B}}
\newcommand{\sC}{\mathscr{C}}
\newcommand{\sD}{\mathscr{D}}
\newcommand{\sE}{\mathscr{E}}
\newcommand{\sF}{\mathscr{F}}
\newcommand{\sG}{\mathscr{G}}
\newcommand{\sH}{\mathscr{H}}
\newcommand{\sI}{\mathscr{I}}
\newcommand{\sJ}{\mathscr{J}}
\newcommand{\sK}{\mathscr{K}}
\newcommand{\sL}{\mathscr{L}}
\newcommand{\sM}{\mathscr{M}}
\newcommand{\sN}{\mathscr{N}}
\newcommand{\sO}{\mathscr{O}}
\newcommand{\sP}{\mathscr{P}}
\newcommand{\sQ}{\mathscr{Q}}
\newcommand{\sR}{\mathscr{R}}
\newcommand{\sS}{\mathscr{S}}
\newcommand{\sT}{\mathscr{T}}
\newcommand{\sU}{\mathscr{U}}
\newcommand{\sV}{\mathscr{V}}
\newcommand{\sW}{\mathscr{W}}
\newcommand{\sX}{\mathscr{X}}
\newcommand{\sY}{\mathscr{Y}}
\newcommand{\sZ}{\mathscr{Z}}
\newcommand{\bZ}{\mathbb{Z}}
\newcommand{\bN}{\mathbb{N}}
\newcommand{\bQ}{\mathbb{Q}}
\newcommand{\bC}{\mathbb{C}}
\newcommand{\bR}{\mathbb{R}}
\newcommand{\bH}{\mathbb{H}}
\newcommand{\bD}{\mathbb{D}}
\newcommand{\bE}{\mathbb{E}}
\newcommand{\bV}{\mathbb{V}}
\newcommand{\cV}{\mathcal{V}}
\newcommand{\cF}{\mathcal{F}}
\newcommand{\bfM}{\mathbf{M}}
\newcommand{\bfN}{\mathbf{N}}
\newcommand{\bfX}{\mathbf{X}}
\newcommand{\bfY}{\mathbf{Y}}
\newcommand{\spec}{\textrm{Spec}}
\newcommand{\dbar}{\bar{\partial}}
\newcommand{\ddbar}{\partial\bar{\partial}}
\newcommand{\redref}{{\color{red}ref}}
%%%%%%%%%%%%%%%%%%%%%%%%%%%%%%%%%%%%%%%%%%%%%%%%%%%%%%%%%%%%%%

\title[$L^2$-Extension of Adjoint bundles and Koll\'ar's Conjecture] {$L^2$-Extension of Adjoint bundles and Koll\'ar's Conjecture}

\author[Junchao Shentu]{Junchao Shentu}
\email{stjc@ustc.edu.cn}
\address{School of Mathematical Sciences,
	University of Science and Technology of China, Hefei, 230026, China}
\author[Chen Zhao]{Chen Zhao}
\email{czhao@ustc.edu.cn}
\address{School of Mathematical Sciences,
	University of Science and Technology of China, Hefei, 230026, China}

%%%%%%%%%%%%%%%%%%%%%%%%%%%%%%%%%%%%%%%%%%%%%%%%%%%%%%%%%%%%%%%%%%%%%%%%%%%%%%%%%%%%%%%%%%%%%%%%
%                                          Abstract                                            %
%%%%%%%%%%%%%%%%%%%%%%%%%%%%%%%%%%%%%%%%%%%%%%%%%%%%%%%%%%%%%%%%%%%%%%%%%%%%%%%%%%%%%%%%%%%%%%%%
\begin{abstract}
We give a new proof of Koll\'ar's conjecture on the pushforward of the dualizing sheaf twisted by a variation of Hodge structure. This conjecture was settled by M. Saito via mixed Hodge modules and has applications in the investigation of Albanese maps. Our technique is the $L^2$-method and we give a concrete construction and proofs of the conjecture. The $L^2$ point of view allows us to generalize Koll\'ar's conjecture to the context of non-abelian Hodge theory.
\end{abstract}

\maketitle
%\tableofcontents

\section{Introduction}
Let $f:X\rightarrow Y$ be a surjective morphism between complex projective varieties. Assume that $X$ is smooth and denote by $\omega_X$ its dualizing sheaf. In \cite{Kollar1986_1,Kollar1986_2}, J. Koll\'ar proves the following results which is roughly called the Koll\'ar package in this paper.
\begin{description}
	\item[Torsion Freeness] $R^if_\ast\omega_X$ is torsion free for $i\geq 0$ and $R^if_\ast\omega_X=0$ if $i>\dim X-\dim Y$.
	\item[Vanishing Theorem] If $L$ is an ample line bundle on $Y$, then
	$$H^j(Y,R^if_\ast\omega_X\otimes L)=0\textrm{ for }\forall j>0 \textrm{ and }\forall i\geq0.$$
	\item[Decomposition Theorem] $Rf_\ast(\omega_X)$ splits in $D(Y)$, i.e. 
	$$Rf_\ast(\omega_X)\simeq \bigoplus_{q} R^qf_\ast(\omega_X)[-q]\in D(Y).$$ As a consequence, the spectral sequence
	$$E^{pq}_2:=H^p(Y,R^qf_\ast\omega_X)\Rightarrow H^{p+q}(X,\omega_X)$$
	degenerates at the $E_2$ page.
\end{description}
Motivated by the proofs, Koll\'ar \cite[\S 5]{Kollar1986_2} conjectured that the Koll\'ar package could be put into a more general framework which is closely related to  variations of Hodge structure. More precisely, Koll\'ar conjectured that there is a coherent sheaf $S_X(\bV)$ associated to every polarized variation of Hodge structure $\bV=(\cV,\nabla,\cF^\bullet,h_\bV)$ over some dense Zariski open subset of $X_{\rm reg}$, such that the three results above hold when $\omega_X$ is replaced by $S_X(\bV)$. This conjecture is perfectly settled in \cite{MSaito1991} by M. Saito's theory of mixed Hodge modules \cite{MSaito1988,MSaito1990} and has applications in the investigation of Albanese maps.

The purpose of this paper has two sides.
\begin{enumerate}
	\item Give a concrete construction of $S_X(\bV)$ by using certain $L^2$ holomorphic sections and reprove Koll\'ar's conjecture without using mixed Hodge modules. There are, in addition to the concrete construction, two other advantages of the $L^2$ method:
	\begin{enumerate}
		\item It allows us to prove Koll\'ar's conjecture for proper K\"ahler morphisms. The machinery of mixed Hodge modules does not work for this more general setting because the decomposition theorem of Saito with respect to a proper K\"ahler morphism is still a conjecture (\cite[Conjecture 0.4]{MSaito1990(3)}).
		\item It allows us to prove Koll\'ar's conjecture with a twisted coefficients in a hermitian vector bundle with a Nakano semi-positive curvature. This flexibility is convenient for applications.
	\end{enumerate}
	\item We observe that, rather than the structure of the variation of Hodge structure, the validity of the Koll\'ar package for $S_X(\bV)$ is a consequence of the Nakano semipositivity of the top Hodge bundle $S(\bV):=\cF^{\max\{k|\cF^k\neq0\}}$ (see \S \ref{section_abstract_Kollar_package} for the abstract Koll\'ar package). This allows us to further generalize Kollar's conjecture to the context of non-abelian Hodge theory. For example, we show that the Koll\'ar package holds for $S_X(E,h)$ (see below for definition) when $E$ is a subbundle of a tame harmonic bundle $(H,\theta,h)$ such that $\overline{\theta}(E)=0$ and the second fundamental form of $E\subset H$ vanishes. A typical example is a system of Hodge bundles in the sense of C. Simpson \cite{Simpson1988} on some Zariski open subset $X^o\subset X_{\rm reg}$ while $E$ is the top nonzero Hodge summand. This is an example that supports the principle that many results for variation of Hodge structure hold for harmonic bundles. Readers may see the hard Lefschetz theorems of Simpson \cite{Simpson1992}, C. Sabbah \cite{Sabbah2005} and T. Mochizuki \cite{Mochizuki20072,Mochizuki20071}, the vanishing theorems of D. Arapura \cite{Arapura2019} and Y. Deng-F. Hao \cite{Dengya2019} and Mochizuki's degeneration theory for twistor structures \cite{Mochizuki20072} for other examples supporting this principle. Notice that the Koll\'ar conjecture for complex variations of Hodge structure may be proved in the patten of Saito \cite{MSaito1991} by using the theory of complex Hodge modules (see \cite{Popa-Schnell2017} or the MHM Project by C. Sabbah and C. Schnell on the homepage of Sabbah.). 
\end{enumerate}
Let $X$ be a complex space and $X^o\subset X_{\rm reg}$ a dense Zariski open subset. Let $(E,h)$ be a hermitian vector bundle on $X^o$. Denote by $j:X^o\to X$ the open immersion. The main object of the present paper is the subsheaf $S_X(E,h)\subset j_\ast(K_{X^o}\otimes E)$ consisting of the holomorphic forms which are locally square integrable at every point of $X$.
This class of objects contains Saito's $S$-sheaf and shares many properties of the dualizing sheaf. It has the following features:
\begin{enumerate}
	\item $S_X(E,h)$ is a torsion free $\sO_X$-module. It is coherent when $(E,h)$ is Nakano semi-positive and tame (Definition \ref{defn_tame_hermitian_bundle}).
	\item $S_X(E,h)$ has the functorial property (Proposition \ref{prop_L2ext_birational}): Let $\pi:X'\to X$ be a proper bimeromorphic map such that $\pi|_{\pi^{-1}X^o}:\pi^{-1}X^o\to X^o$ is biholomorphic, then 
	\begin{align*}
	S_X(E,h)\simeq \pi_\ast(S_{X'}(\pi^\ast E,\pi^\ast h)).
	\end{align*}
	\item Let $\bV:=(\cV,\nabla,\cF^\bullet,h_\bV)$ be an $\bR$-polarized variation of Hodge structure where $h_\bV$ is the Hodge metric. Denote by $S(\bV):=\cF^{\max\{k|\cF^k\neq0\}}$ the top nonzero Hodge piece and by $S(IC_X(\bV))$ the Saito's $S$-sheaf associated to the Hodge module $IC_X(\bV)$. Then $$S_X(S(\bV),h_\bV)\simeq S(IC_X(\bV))\quad \textrm{(\cite[Theorem 4.10]{SC2021})}.$$
	\item Consider a tame harmonic bundle $(H,\theta,h)$ and a holomorphic subbundle $E\subset H$ with vanishing second fundamental form. Assume that $\overline{\theta}(E)=0$. Then the corresponding $S_X(E,h)$ is a coherent sheaf (Proposition \ref{prop_CVHS_Ssheaf}) and satisfies the Koll\'ar package (Theorem \ref{thm_main}). When $(E,\theta,h)$ is the harmonic bundle associated to an $\bR$-polarized variation of Hodge structure $\bV:=(\cV,\nabla,\cF^\bullet,h_\bV)$ and $E=S(\bV)$, this is reduced to Saito's $S(IC_X(\bV))$.
\end{enumerate}
All complex spaces are assumed to be separated, reduced, paracompact, countable at infinity and of pure dimension throughout the present paper. We would like to point out that the complex spaces are allowed to be non-irreducible.
The main result of the present paper is the following 
\begin{thm}\label{thm_main}
	Let $f:X\rightarrow Y$ be a proper locally K\"ahler morphism (Definition \ref{defn_loc_Kahler}) from a complex space $X$ to an irreducible complex space $Y$. Assume that every irreducible component of $X$ is mapped onto $Y$. Let $X^o\subset X_{\rm reg}$ be a dense Zariski open subset and $(H,\theta,h)$ a tame harmonic bundle on $X^o$. Let $E\subset H$ be a holomorphic subbundle with vanishing second fundamental form. Let $\overline{\theta}$ be the adjoint of $\theta$ and assume that $\overline{\theta}(E)=0$. Let $F$ be a Nakano semi-positive vector bundle on $X$. Then the following statements hold.
	\begin{description}
		\item[Torsion Freeness] 
			$R^qf_\ast (S_X(E,h)\otimes F)$ is torsion free for every $q\geq0$ and vanishes if $q>\dim X-\dim Y$.
		\item[Injectivity] If $L$ is a semi-positive holomorphic line bundle so that $L^l$ admits a nonzero holomorphic global section $s$ for some $l>0$, then the canonical morphism
		$$R^qf_\ast(\times s):R^qf_\ast(S_X(E,h)\otimes F\otimes L^{\otimes k})\to R^qf_\ast(S_X(E,h)\otimes F\otimes L^{\otimes k+l})$$
		is injective for every $q\geq0$ and every $k\geq1$
		\item[Vanishing] If $Y$ is a projective algebraic variety and $L$ is an ample line bundle on $Y$, then
		$$H^q(Y,R^pf_\ast (S_X(E,h)\otimes F)\otimes L)=0,\quad \forall q>0,\forall p\geq0.$$
		\item[Decomposition] Assume moreover that $X$ is a compact K\"ahler space, then $Rf_\ast (S_X(E,h)\otimes F)$ splits in $D(Y)$, i.e. 
		$$Rf_\ast(S_X(E,h)\otimes F)\simeq \bigoplus_{q} R^qf_\ast(S_X(E,h)\otimes F)[-q]\in D(Y).$$
		As a consequence, the spectral sequence
		$$E^{pq}_2:H^p(Y,R^qf_\ast(S_X(E,h)\otimes F))\Rightarrow H^{p+q}(X,S_X(E,h)\otimes F)$$
		degenerates at the $E_2$ page.
	\end{description}
\end{thm}
When $(H,\theta,h)$ is a harmonic bundle associated to an $\bR$-polarized variation of Hodge structure $\bV:=(\cV,\nabla,\cF^\bullet,h_\bV)$, $E=S(\bV)$ and $F=\sO_X$, Theorem \ref{thm_main} implies Koll\'ar's conjecture. In this case, our construction of $S_X(\bV)$ coincides with Saito's.

The present paper is organized as follows. In Section 2 we introduce the construction of the adjoint $L^2$-extension $S_X(E,h)$ of a hermitian bundle $(E,h)$ which  generalizes Saito's $S$-sheaf. Some fundamental properties of $S_X(E,h)$ are proved and its $L^2$-Dolbeault resolution is introduced. Depending on the $L^2$-Dolbeault resolution, we introduce a harmonic representation of the derived pushforwards of $S_X(E,h)$ in Section 3 by generalizing K. Takegoshi's work \cite{Takegoshi1995} to complex spaces. In Section 4 we prove an abstract Koll\'ar package and illustrate the observation that: Koll\'ar's conjecture is a consequence of the Nakano semi-positivity of the top Hodge piece.  We prove the main theorem \ref{thm_main} in Section 5 as an application.

{\bf Notation:} Let $X$ be a complex space. A Zariski closed subset (=closed analytic subset) $Z\subset X$ is a closed subset which is locally defined as the zeros of a set of holomorphic functions. A subset $Y\subset X$ is called Zariski open if $X\backslash Z\subset X$ is Zariski closed.
\section{$L^2$-extension and its $L^2$-Dolbeault resolution}
\subsection{$L^2$-Dolbeault cohomology and $L^2$-Dolbeault complex}
A psh (resp. strictly psh) function on a complex space $X$ is a function $\lambda:X\to[-\infty,\infty)$ such that, locally at every point $x\in X$, there is a neighborhood $x\in U$, a closed immersion $\iota:U\to\Omega$ into a complex manifold $\Omega$ and a  psh (resp. strictly psh) function $\Lambda$ on $\Omega$ such that $\iota^\ast\Lambda=\lambda$. By a $C^\infty$ form on $X$ we mean a $C^\infty$ form $\alpha$ on $X_{\rm reg}$ so that locally at every point $x\in X$ there is an open neighborhood $x\in U$, a closed immersion $\iota:U\to\Omega$ into a holomorphic manifold $\Omega$ and $\beta\in C^\infty(\Omega)$ such that $\iota^\ast\beta=\alpha$ on $U\cap X_{\rm reg}$.

Let $(Y,ds^2)$ be a hermitian manifold of dimension $n$ and $(E,h)$ a hermitian vector bundle on $Y$. Let $L^{p,q}_{(2)}(Y,E;ds^2,h)$ be the space of measurable $E$-valued $(p,q)$-forms on $Y$ which are square integrable with respect to the metrics $ds^2$ and $h$. Denote $\dbar_{\rm max}$ to be the maximal extension of the $\dbar$ operator defined on the domains
$$D^{p,q}_{\rm max}(Y,E;ds^2,h):=\textrm{Dom}^{p,q}(\dbar_{\rm max})=\{\phi\in L_{(2)}^{p,q}(Y,E;ds^2,h)|\dbar\phi\in L_{(2)}^{p,q+1}(Y,E;ds^2,h)\}.$$
The $L^2$ cohomology $H_{(2),\rm max}^{p,\bullet}(Y,E;ds^2,h)$ is defined as the cohomology of the complex
\begin{align}\label{align_L2_dol_cohomology}
D^{p,\bullet}_{\rm max}(Y,E;ds^2,h):=D^{p,0}_{\rm max}(Y,E;ds^2,h)\stackrel{\dbar_{\rm max}}{\to}\cdots\stackrel{\dbar_{\rm max}}{\to}D^{p,n}_{\rm max}(Y,E;ds^2,h).
\end{align}

Let $X$ be a complex space and $X^o\subset X_{\rm reg}$ a dense Zariski open subset of the regular locus $X_{\rm reg}$. Let $ds^2$ be a hermitian metric on $X^o$ and $(E,h)$ a hermitian vector bundle on $X^o$.
Let $U\subset X$ be an open subset. Define $L_{X,ds^2}^{p,q}(E,h)(U)$ to be the space of measurable $E$-valued $(p,q)$-forms $\alpha$ on $U\cap X^o$ such that for every point $x\in U$, there is a neighborhood $x\in V_x$ so that 
$$\int_{V_x\cap X^o}|\alpha|^2_{ds^2,h}{\rm vol}_{ds^2}<\infty.$$
For each $0\leq p,q\leq n$, we define a sheaf $\sD_{X,ds^2}^{p,q}(E,h)$ on $X$ by
$$\sD_{X,ds^2}^{p,q}(E,h)(U):=\{\phi\in L_{X,ds^2}^{p,q}(E,h)(U)|\bar{\partial}_{\rm max}\phi\in L_{X,ds^2}^{p,q+1}(E,h)(U)\}$$
for every open subset $U\subset X$.

Define the $L^2$-Dolbeault complex of sheaves $\sD_{X,ds^2}^{p,\bullet}(E,h)$ as
\begin{align}\label{align_D_complex2}
\sD_{X,ds^2}^{p,0}(E,h)\stackrel{\dbar}{\to}\sD_{X,ds^2}^{p,1}(E,h)\stackrel{\dbar}{\to}\cdots\stackrel{\dbar}{\to}\sD_{X,ds^2}^{p,n}(E,h)
\end{align}
where $\dbar$ is defined in the sense of distribution.
\begin{defn}\label{defn_hermitian_metric}
	Let $X$ be a complex space and $ds^2$ a hermitian metric on $X_{\rm reg}$. We say that $ds^2$ is a hermitian metric on $X$ if, for every $x\in X$, there is a neighborhood $U$ and a holomorphic closed immersion $U\subset V$ into a complex manifold $V$ such that $ds^2|_U\sim ds^2_V|_{U}$ for some hermitian metric $ds^2_V$ on $V$. If $ds^2_0|_{X_{\rm reg}}$ is moreover a K\"ahler metric, we say that $ds^2_0$ is a K\"ahler hermitian metric.
\end{defn}
\begin{lem}\label{lem_fine_sheaf}
	Let $X$ be a complex space and $X^o\subset X_{\rm reg}$ a dense Zariski open subset. Let $ds^2$ be a hermitian metric on $X^o$ and $(E,h)$ a hermitian vector bundle on $X^o$. Suppose that for every point $x\in X\backslash X^o$ there is a neighborhood $x\in U_x$ and a hermitian metric $ds^2_0$ on $U_x$ such that $ds^2_0|_{X^o\cap U_x}\lesssim ds^2|_{X^o\cap U_x}$. Then $\sD^{p,q}_{X,ds^2}(E,h)$ is a fine sheaf for every $p$ and $q$.
\end{lem}
\begin{proof}
	If suffices to show that for every $W\subset\overline{W}\subset U\subset X$ where $W$ and $U$ are small open subsets, there is a positive continuous function $f$ on $U$ such that
	\begin{itemize}
		\item ${\rm supp}(f)\subset \overline{W}$,
		\item $f$ is $C^\infty$ on $U\cap X^o$
		\item $\dbar f$ has bounded fiberwise norm, with respect to the metric $ds^2$.
	\end{itemize}
	Choose a closed embedding $U\subset M$ where $M$ is a smooth complex manifold $M$. Let $V\subset\overline{V}\subset M$ where $V$ is an open subset such that $V\cap U=W$. Let $ds^2_M$ be a hermitian metric on $M$ so that $ds^2_0|_{U\cap X^o}\sim ds^2_M|_{U\cap X^o}$. Let $g$ be a positive smooth function on $M$ whose support lies in $\overline{V}$. Denote $f=g|_U$, then apparently ${\rm supp}(f)\subset \overline{W}$ and $f$ is $C^\infty$ on $U\cap X^o$. It suffices to show the boundedness of the fiberwise norm of $\dbar f$. Since $U\cap X^o\subset M$ is a submanifold, one has the orthogonal decomposition 
	\begin{align}
	T_{M,x}=T_{U\cap X^o,x}\oplus T_{U\cap X^o,x}^\bot, \quad \forall x\in U\cap X^o.
	\end{align} 
	Therefore $|\dbar f|_{ds^2}\lesssim|\dbar f|_{ds^2_0}\leq |\dbar g|_{ds^2_M}<\infty$. The lemma is proved.
\end{proof}
\subsection{$S_X(E,h)$ and its basic properties}
Let $X$ be a complex space of dimension $n$ and $X^o\subset X_{\rm reg}$ a dense Zariski open subset. Let $ds^2$ be a hermitian metric on $X^o$ and $(E,h)$ a hermitian vector bundle on $X^o$.
\begin{defn}
	Define
	\begin{align}
	S_X(E,h):={\rm Ker}\left(\sD_{X,ds^2}^{n,0}(E,h)\stackrel{\dbar}{\to}\sD_{X,ds^2}^{n,1}(E,h)\right).
	\end{align}
\end{defn}
The following proposition shows that $S_X(E,h)$ is independent of $ds^2$. Hence $ds^2$ is omitted in the notation $S_X(E,h)$.
\begin{prop}\label{prop_L2_extension_independent}
	$S_X(E,h)$ is independent of $ds^2$. 
\end{prop}

\begin{proof}
	Let $\pi:\tilde{X}\to X$ be a desingularization of $X\backslash X^o$ so that $\widetilde{X}$ is smooth and $\pi$ is biholomorphic over $X^o$. Let $ds^2_{\tilde{X}}$ be a hermitian metric on $\tilde{X}$. Since $\pi$ is a proper map, a section of $K_{X^o}\otimes E$ is locally square integrable at $x$ if and only if it is locally square integrable near $\pi^{-1}\{x\}$. Thus it suffices to show that 
	\begin{align}\label{align_S_independent}
	{\rm Ker}\left(\sD_{\tilde{X},\pi^\ast ds^2}^{n,0}(E,h)\stackrel{\dbar}{\to}\sD_{\tilde{X},\pi^\ast ds^2}^{n,1}(E,h)\right)={\rm Ker}\left(\sD_{\tilde{X},ds^2_{\tilde{X}}}^{n,0}(E,h)\stackrel{\dbar}{\to}\sD_{\tilde{X},ds^2_{\tilde{X}}}^{n,1}(E,h)\right).
	\end{align}	
	Since the problem is local, we assume that there is an orthogonal frame of cotangent fields $\delta_1,\dots,\delta_n$ such that 
	\begin{align}\label{align_lem_metric1}
	\pi^\ast ds^2\sim\lambda_1\delta_1\overline{\delta_1}+\cdots+\lambda_n\delta_n\overline{\delta_n}
	\end{align}
	and
	\begin{align}\label{align_lem_metric2}
	ds^2_{\tilde{X}}\sim\delta_1\overline{\delta_1}+\cdots+\delta_n\overline{\delta_n}.
	\end{align}
	Let $s=\delta_1\wedge\cdots\wedge\delta_n\otimes\xi$.
	By (\ref{align_lem_metric1}) and (\ref{align_lem_metric2}) we obtain
	\begin{align}\label{align_ind_ds2}
	\|s\|^2_{\pi^\ast ds^2,h}&=\int|\delta_1\wedge\cdots\wedge\delta_n\otimes\xi|^2_{\pi^\ast ds^2,h}\prod_{i=1}^n\lambda_i\delta_i\wedge\overline{\delta_i}\\\nonumber
	&=\int|\xi|^2_{h}\prod_{i=1}^n\delta_i\wedge\overline{\delta_i}\\\nonumber
	&=\|s\|^2_{ds^2_{\tilde{X}},h}.
	\end{align}
	Therefore $\|s\|^2_{\pi^\ast ds^2,h}$ is locally finite if and only if $\|s\|^2_{ds^2_{\tilde{X}},h}$ is locally finite. This proves (\ref{align_S_independent}).
\end{proof}
\begin{prop}[Functoriality]\label{prop_L2ext_birational}
	Let $\pi:X'\to X$ be a proper holomorphic map between complex spaces which is biholomorphic over $X^o$. Then $$\pi_\ast S_{X'}(\pi^\ast E,\pi^\ast h)=S_X(E,h).$$
\end{prop}
\begin{proof}
	It follows from Proposition \ref{prop_L2_extension_independent} that 
	$$S_{X'}(\pi^\ast E,\pi^\ast h)={\rm Ker}\left(\sD_{X',\pi^\ast ds^2}^{n,0}(\pi^\ast E,\pi^\ast h)\stackrel{\dbar}{\to}\sD_{X',\pi^\ast ds^2}^{n,1}(\pi^\ast E,\pi^\ast h)\right)$$
	and
	$$S_{X}(E,h)={\rm Ker}\left(\sD_{X, ds^2}^{n,0}(E,h)\stackrel{\dbar}{\to}\sD_{X, ds^2}^{n,1}(E,h)\right).$$
	Since $\pi$ is a proper map, a section of $E$ is locally square integrable at $x$ if and only if it is locally square integrable near $\pi^{-1}\{x\}$. This proves the lemma.
\end{proof}
The following simple lemma is important for applications. Roughly speaking, it says that one can shrink the domain of $(E,h)$ without changing $S_X(E,h)$. This phenomenon also appears in Saito's $S(M)$.
\begin{lem}
	Let $U\subset X^o$ be a dense Zariski open subset. Then
	$$S_X(E,h)\simeq S_X(E|_U,h|_U).$$
\end{lem}
\begin{proof}
	This is a consequence of the fact that if a locally $L^2$ function on a hermitian manifold is $\dbar$-closed away from a Zariski open subset, then it is $\dbar$-closed over the whole manifold (\cite[Lemma 1.3]{Berndtsson2010}).
\end{proof}
\begin{lem}\label{lem_Kernel}
	Let $(F,h_F)$ be a hermitian vector bundle on $X$, then
	$$S_X(E,h)\otimes F\simeq S_X(E\otimes F,h\otimes h_{F}).$$
\end{lem}
\begin{proof}
	Let $x\in X$ be a point and let $U$ be an open neighborhood of $x$ so that $F|_U\simeq \sO_U^{\oplus r}$ and $h_F$ is quasi-isometric to the trivial metric, i.e.
	$$|\sum_{i=1}^ra_ie_i|^2_{h_F}\sim \sum_{i=1}^r\|a_i\|^2$$
	where $\{e_i\}$ is the standard frame of $\sO_U^{\oplus r}$ and $a_i$s are measurable functions on $U\cap X^o$.
	Let $ds^2$ be an arbitrary hermitian metric on $X^o$ and let $\alpha=\sum_{i=1}^r\alpha_i\otimes e_i$ be a measurable section of $(K_{X^o}\otimes E\otimes F)|_{U\cap X^o}$. Then
	$$\|\alpha\|^2_{ds^2,h\otimes h_F}\sim\sum_{i=1}^r\|\alpha_i\|^2_{ds^2,h}$$
	is finite if and only if $\|\alpha_i\|^2_{ds^2,h}$ is finite for each $i=1,\dots,r$. This proves the lemma.
\end{proof}

\iffalse
\begin{prop}[Minimal Extension Property]
	Assume that $(E,h)$ has a Nakano semi-positive curvature, then $S_X(E,h)$ satisfies the ''minimal extension property'': For every embedding $\iota:\Delta:=\{z\in\bC|\|z\|<1\}\to X$ with $\iota(0)\in X^o$ and every $v\in E_{\iota(0)}$ with $|v|_{h}=1$, there is a section $\sigma\in H^o(\Delta,\iota^\ast S_X(E,h))$ such that $\sigma(0)=v$ and 
	$$\int_\Delta|\sigma|^2{\rm vol}_{ds^2}\leq {\rm Vol}(\Delta).$$
\end{prop}

\begin{proof}
	This is a consequence of the Ohsawa-Takegoshi extension theorem (\cite{OT1988,Guan-Zhou2015}, see also \cite[Theorem 2.12]{Schnell2020}).
\end{proof}
\fi

At the end of this section, we show that $S_X(E,h)$ is a coherent sheaf if the dual metric $h^\ast$ has at most polynomially growth at the boundary $X\backslash X^o$ and $(E,h)$ is Nakano semi-positive.

Let $Y$ be a complex manifold and $E$ a holomorphic vector bundle on $Y$. Let $\Theta\in A^{1,1}(Y, {\rm End}(E))$. Denote $\Theta\geq0$ if $\Theta$ is Nakano semi-positive, i.e.
\begin{align}
\Theta(v,v)\geq0,\quad\forall v\in \Gamma(T_Y\otimes E).
\end{align}
Let $\Theta_1,\Theta_2\in A^{1,1}(Y, {\rm End}(E))$, then $\Theta_1\geq\Theta_2$ stands for $\Theta_1-\Theta_2\geq0$. A hermitian vector bundle $(E,h)$ is Nakano semi-positive if $\sqrt{-1}\Theta_h(E)\geq0$.
\begin{defn}\label{defn_tame_hermitian_bundle}
	$(E,h)$ has tame singularity on $X$ if, for every point $x\in X$, there is an open neighborhood $U$, a proper bimeromorphic morphism $\pi:\widetilde{U}\to U$ which is biholomorphic over $U\cap X^o$, and a hermitian vector bundle $(Q,h_Q)$ on $\widetilde{U}$ such that 
	\begin{enumerate}
		\item $\pi^\ast E|_{\pi^{-1}(X^o\cap U)}\subset Q|_{\pi^{-1}(X^o\cap U)}$ as a subsheaf.
		\item There is a hermitian metric $h'_Q$ on $Q|_{\pi^{-1}(X^o\cap U)}$ so that $h'_Q|_{\pi^\ast E}\sim \pi^\ast h$ on $\pi^{-1}(X^o\cap U)$ and
		\begin{align}\label{align_tame}
		(\sum_{i=1}^r\|\pi^\ast f_i\|^2)^ch_Q\lesssim h'_Q
		\end{align}
		for some $c\in\bR$. Here $(f_1,\dots,f_r)$ is an arbitrary local generator of the ideal sheaf defining $\widetilde{U}\backslash \pi^{-1}X^o\subset \widetilde{U}$.
	\end{enumerate}
\end{defn}
The tameness condition (\ref{align_tame}) is independent of the choice of the local generator. In the present paper, a tame hermitian vector bundle $(E,h)$ is constructed as a subsheaf of a tame harmonic bundle (see \S \ref{section_CVHS_Kollar}, especially Proposition \ref{prop_CVHS_Ssheaf}) in the sense of Simpson \cite{Simpson1990} and T. Mochizuki \cite{Mochizuki20072,Mochizuki20071}. In this case, condition (\ref{align_tame}) comes from the tameness condition of the harmonic bundle. This is the origin of the name "tame hermitian vector bundle".

\begin{prop}\label{prop_S_coherent}
	$S_X(E,h)$ is a coherent sheaf if $(E,h)$ is Nakano semi-positive and tame on $X$.
\end{prop}
\begin{proof}
	Since the problem is local, we assume that $X$ is a germ of complex space.
	Let $\pi:\widetilde{X}\to X$ be a desingularization so that $\pi$ is biholomorphic over $X^o$ and $D:=\pi^{-1}(X\backslash X^o)$ is a simple normal crossing divisor. By abuse of notations we regard $X^o\subset \widetilde{X}$ as a subset. Since $(E,h)$ is tame, we assume that there is a hermitian vector bundle $(Q,h_Q)$ on $\widetilde{X}$ such that $E$ is a subsheaf of $Q|_{X^o}$ and there exists $m\in \bN$ such that 
	\begin{align}\label{align_tame_1}
	\|z_1\cdots z_r\|^{2m}h_Q|_{E}\lesssim \pi^\ast h
	\end{align}
	where $(z_1,\cdots,z_n)$ is a local coordinate of $\widetilde{X}$ so that $D=\{z_1\cdots z_r=0\}$. There is moreover a hermitian metric $h'_Q$ on $Q|_{X^o}$ so that $h'_Q|_{E}\sim h$.
    By Proposition \ref{prop_L2ext_birational} there is an isomorphism
	\begin{align}
	S_X(E,h)\simeq \pi_\ast\left(S_{\widetilde{X}}(E,h)\right).
	\end{align}
	Since $\pi$ is a proper map, it suffices to show that $S_{\widetilde{X}}(E,h)$ is a coherent sheaf on $\widetilde{X}$. Since the problem is local and $\widetilde{X}$ is smooth, we may assume that $\widetilde{X}\subset\bC^n$ is the unit ball so that $D=\{z_1\cdots z_r=0\}$. Without loss of generality we assume that $Q$ admits a global holomorphic frame $e_1,\dots,e_{l}$ which is orthonormal with respect to $h_Q$, i.e.
	\begin{align}\label{align_orthogonal_frame}
	\langle e_i,e_j\rangle_{h_Q}=\begin{cases}
	1, & i=j \\
	0, & i\neq j
	\end{cases}.
	\end{align}
There is a complete K\"ahler metric on $X^o$ by Lemma \ref{lem_complete_metric_exists_locally}. Since $Q$ is coherent\footnote{We do not distinguish holomorphic vector bundle and its sheaf of holomorphic sections.},  the space $\Gamma(\widetilde{X},S_{\widetilde{X}}(E,h))$ generates a coherent subsheaf $\sJ$ of $Q$ by strong Noetherian property of coherent sheaves. We have the inclusion 
	$\sJ\subset S_{\widetilde{X}}(E,h)$ by the construction. It remains to prove the converse. By the Krull's theorem (\cite[Corollary 10.19]{Atiyah1969}), it suffices to show that
	\begin{align}\label{align_Krull}
	\sJ_x+S_{\widetilde{X}}(E,h)_x\cap m_{\widetilde{X},x}^{k+1}Q=S_{\widetilde{X}}(E,h)_x,\quad\forall k\geq0,\forall x\in\widetilde{X}.
	\end{align}
	Let $\alpha\in S_{\widetilde{X}}(E,h)_x$ be defined in a precompact neighborhood $V$ of $x$. Choose a $C^\infty$ cut-off function $\lambda$ so that $\lambda\equiv1$ near $x$ and ${\rm supp}\lambda\subset V$. Denote $\|z\|^2:=\sum_{i=1}^n\|z_i\|^2$. 
	Let
	\begin{align}
	\psi_k(z):=2(n+k+rm)\log|z-x|+\|z\|^2
	\end{align}
	and $h_{\psi_k}=e^{-\psi_k}h$. Denote $\omega_0:=\sqrt{-1}\ddbar\|z\|^2$, then
	\begin{align}
	\sqrt{-1}\Theta_{h_{\psi_k}}(E)= \sqrt{-1}\ddbar\psi_k+\sqrt{-1}\Theta_{h}(E)\geq \omega_0.
	\end{align}
	Since ${\rm supp}\lambda\alpha\subset V$ and $\dbar(\lambda\alpha)=0$ near $x$, we know that
	\begin{align}
	\|\dbar(\lambda\alpha)\|^2_{\omega_0,h_{\psi_k}}\sim \|\dbar(\lambda\alpha)\|^2_{\omega_0,h}\leq\|\dbar\lambda\|^2_{L^\infty}\|\alpha\|^2_{\omega_0,h}+\|\lambda\|^2\|\dbar\alpha\|^2_{\omega_0,h}<\infty
	\end{align}
	Hence Theorem \ref{thm_Hormander_incomplete} gives a solution of the equation $\dbar\beta=\dbar(\lambda\alpha)$ so that
	\begin{align}\label{align_norm_beta_psi}
	\|\beta\|^2_{\omega_0,h}\lesssim\int_{X^o}|\beta|^2_{\omega_0,h}|z-x|^{-2(n+k+rm)}{\rm vol}_{\omega_0}\lesssim \|\dbar(\lambda\alpha)\|^2_{\omega_0,h_{\psi_k}}<\infty.
	\end{align}
	Thus $\gamma=\beta-\lambda\alpha$ is holomorphic and $\gamma\in \Gamma(\widetilde{X},S_{\widetilde{X}}(E,h))$. 
	
    Assume that
	$$\beta=\sum_{i=1}^l f_ie_idz_1\wedge\cdots\wedge dz_n$$
	for some holormorphic functions $f_1,\dots,f_l\in\sO_{\widetilde{X}}(X^o)$.
	By (\ref{align_tame_1}) we obtain that 
	\begin{align}
	&\sum_{i=1}^l\int_{X^o}|f_i|^2\|z_1\cdots z_r\|^{2m}|z-x|^{-2(n+k+rm)}{\rm vol}_{\omega_0}\\\nonumber
	=&\int_{X^o}|\beta|^2_{\omega_0,h_Q}\|z_1\cdots z_r\|^{2m}|z-x|^{-2(n+k+rm)}{\rm vol}_{\omega_0}\\\nonumber
	\lesssim&\int_{X^o}|\beta|^2_{\omega_0,h}|z-x|^{-2(n+k+rm)}{\rm vol}_{\omega_0}<\infty.
	\end{align}
	This implies that $z_1^m\cdots z_r^m f_i\in m_{\widetilde{X},x}^{k+1+rm}$ for every $i=1,\dots,l$ (\cite[Lemma 5.6]{Demailly2012}). Therefore $\beta_x\in m_{\widetilde{X},x}^{k+1}Q$ and we prove (\ref{align_Krull}).
\end{proof}

\subsection{$L^2$-Dolbeault resolution of $S_X(E,h)$}
In this subsection we introduce an $L^2$-Dolbeault resolution of $S_X(E,h)$. 

First we recall the following useful  estimate.
\begin{thm}{\cite[Theorem 5.1]{Demailly1982}}\label{thm_Hormander_incomplete}
	Let $Y$ be a complex manifold of dimension $n$ which admits a complete K\"ahler metric. Let $(E,h)$ be a hermitian vector bundle such that $$\sqrt{-1}\Theta_{h}(E)\geq \omega\otimes {\rm Id}_{E}$$ for some (not necessarily complete) K\"ahler form $\omega$ on $Y$. Then for every $q>0$ and every $\alpha\in L^{n,q}_{(2)}(Y,E;\omega,h)$ such that $\dbar\alpha=0$, there is $\beta\in L^{n,q-1}_{(2)}(Y,E;\omega,h)$ such that $\dbar\beta=\alpha$ and $\|\beta\|^2_{\omega,h}\leq q^{-1}\|\alpha\|^2_{\omega,h}$.
\end{thm}

The main theorem of this section is the following
\begin{thm}\label{thm_main_local1}
	Let $X$ be a complex space of dimension $n$ and $ds^2$ a hermitian metric on a dense Zariski open subset $X^o\subset X_{\rm reg}$ with $\omega$ its fundamental form. Let $(E,h)$ be a Nakano semi-positive hermitian vector bundle. Assume that, locally at every point $x\in X$, there is a neighborhood $x\in U$, a strictly psh function $\lambda\in C^2(U)$ and a bounded psh function $\Phi\in C^2(U\cap X^o)$ such that $$\sqrt{-1}\ddbar\lambda|_{U\cap X^o}\lesssim\omega|_{U\cap X^o}\lesssim\sqrt{-1}\ddbar\Phi$$
	Then the canonical map
	\begin{align}\label{align_fine_resolution}
	S_X(E,h)\to \sD^{n,\bullet}_{X,ds^2}(E,h)
	\end{align}
	is a quasi-isomorphism. As a consequence, there is a canonical isomorphism
	$$H^q(X,S_X(E,h))\simeq H^{n,q}_{(2)}(X^o,E;ds^2,h),\quad\forall q\geq0$$
	if $X$ is compact.
\end{thm}
\begin{proof}
	It suffices to show that (\ref{align_fine_resolution}) is exact at $\sD^{n,q}_{X,ds^2}(E,h)$ for $q>0$. Since the problem is local, we consider a point $x\in X$ and an open neighborhood $x\in U$ which is small enough so that there is $C>0$ and a bounded psh function $\Phi\in C^2(U\cap X^o)$ such that $ C\sqrt{-1}\partial\dbar\Phi\geq\omega|_{U\cap X^o}$. Let $h'=e^{-C\Phi}h$. Since $\Phi$ is bounded and $(E,h)$ is Nakano semi-positive, we have $h'\sim h$ and
$$\sqrt{-1}\Theta_{h'}(E|_{U\cap X^o})=C\sqrt{-1}\partial\dbar\Phi\otimes{\rm Id}_E|_{U\cap X^o}+\sqrt{-1}\Theta_{h}(E|_{U\cap X^o})\geq \omega\otimes{\rm Id}_E|_{U\cap X^o}.$$
	By Lemma \ref{lem_complete_metric_exists_locally}, we may assume that $U\cap X^o$ admits a complete K\"ahler metric. By Theorem \ref{thm_Hormander_incomplete}, we therefore have
	$$H^{n,q}_{(2)}(U\cap X^o, E|_{U\cap X^o};ds^2,h)=H^{n,q}_{(2)}(U\cap X^o, E|_{U\cap X^o};ds^2,h')=0,\quad \forall q>0.$$
	This proves the exactness of (\ref{align_fine_resolution}) at $\sD^{n,q}_{X,ds^2}(E,h)$, $q>0$. 
	Since $\omega$ is locally bounded from below by a hermitian metric, $\sD^{n,q}_{X,ds^2}(E,h)$ is a fine sheaf for every $q$ by Lemma \ref{lem_fine_sheaf}. The second claim therefore follows from the compactness of the space $X$. 
\end{proof}
In order to apply Theorem \ref{thm_main_local1}, we introduce a type of hermitian metric that is crucial for the present paper.
\begin{defn}\label{align_various_metric}
	Let $X$ be a complex space and $X^o\subset X_{\rm reg}$ a dense Zariski open subset. Let $ds^2$ be a hermitian metric on $X^o$. 
	\begin{enumerate}
		\item $ds^2$ admits a $(\infty,1)$ bounded potential locally  on $X$ if, for every point $x\in X$, there is a neighborhood $U$ of $x$, a function $\Phi\in C^{\infty}(U)$ such that $|\Phi|+|d\Phi|_{ds^2}<\infty$ and $ds^2|_U\sim\sqrt{-1}\ddbar\Phi$.
		\item $ds^2$ is called locally complete on $X$ if, for every point $x\in X$, there is a neighborhood $U$ of $x$ such that $(\overline{U}\cap X^o,ds^2)$ is complete.
		\item $ds^2$ is locally bounded from below by a hermitian metric if, for every point $x\in X$, there is a neighborhood $U$ of $x$ and a hermitian metric $ds^2_0$ on $U$ such that $ds^2_0|_U\lesssim ds^2|_U$.
	\end{enumerate}
\end{defn}
\begin{lem}\label{lem_complete_metric_exists_locally}
	Let $X$ be a weakly pseudoconvex K\"ahler space with $\psi$ a smooth exhausted psh function on $X$. Denote $X_c:=\{x\in X|\psi(x)<c\}$. Let $X^o\subset X$ be a dense Zariski open subset. Then, for every $c\in\bR$, there exists a complete K\"ahler metric $ds^2$ on $X_c\cap X^o$ such that 
	\begin{enumerate}
		\item $ds^2_0\lesssim ds^2$ for some hermitian metric $ds^2_0$ on $X_c$;
		\item $ds^2$ admits a $(\infty,1)$ bounded potential locally on $X_c$.
	\end{enumerate} 
\end{lem}
\begin{proof}	
	Let $U$ be a neighborhood of a point $x\in X\backslash X^o$. Assume that $U\backslash X^o\subset U$ is defined by $f_1,\dots,f_r\in\sO_U(U)$. Let 
	$$\varphi_U:=\frac{1}{\log(-\log\sum_{i=1}^r\|f_i\|^2)}+\phi_U,$$
	where $\phi_U$ is a strictly $C^\infty$ psh function on $U$ so that $\varphi_U$ is strictly psh.
	Then the quasi-isometric class of $\sqrt{-1}\ddbar\varphi_U$ is independent of the choice of $\{f_1,\dots,f_r\}$ and $\phi_U$. By partition of unity the potential functions $\varphi_U$ can be glued to a global $\varphi$ on $X$ so that 
	\begin{align}\label{align_Grauert_local}
	\sqrt{-1}\ddbar\varphi|_U\sim\sqrt{-1}\ddbar\varphi_U
	\end{align}
	near every point $x\in X\backslash X^o$ and $\varphi\equiv0$ away from a neighborhood $V$ of $X\backslash X^o$.
	
	Denoting $u=-\log\sum_{i=1}^r\|f_i\|^2$,  we assume that $u>e$ on $V$ by a possible shrinking of $V$. Then
	\begin{align}
	\sqrt{-1}\ddbar\varphi|_U&\sim\sqrt{-1}\frac{2+\log u}{u^2\log^3u}\partial u\wedge\dbar u+\sqrt{-1}\frac{\ddbar u}{u\log^2u}+\sqrt{-1}\ddbar\phi_U\\\nonumber
	&\sim\sqrt{-1}\frac{\partial u\wedge\dbar u}{u^2\log^2u}+\sqrt{-1}\frac{\ddbar u}{u\log^2u}+\sqrt{-1}\ddbar\phi_U.
	\end{align}
	Hence 
	$$|\varphi|+|d\varphi|_{\sqrt{-1}\ddbar\varphi}\lesssim \frac{1}{\log u}<1.$$
	Since $\log\log u$ is a smooth exhausted function such that 
	$$|d\log\log u|_{\sqrt{-1}\ddbar\varphi}\leq 2.$$
	By the Hopf-Rinow theorem, $\sqrt{-1}\ddbar\varphi$ is locally complete near $X\backslash X^o$.
	
	Let $c\in\bR$ and let $\omega_0$ be a K\"ahler hermitian metric on $X$. By adding a constant to $\psi$, we assume $\psi\geq 0$. Then $\psi_c:=\psi+\frac{1}{c-\psi}$ is a smooth exhausted psh function on $X_c=\{x\in X|\psi(x)<c\}$. Hence
	$\omega_0+\sqrt{-1}\ddbar\psi^2_c$
	is a complete K\"ahler hermitian metric on $X_c$ (\cite[Theorem 1.3]{Demailly1982}).
	
	Since $\overline{X_c}$ is compact, $$\sqrt{-1}\ddbar\varphi+K(\omega_0+\sqrt{-1}\ddbar\psi^2_c),\quad K\gg0$$
	is positive definite and it gives a  desired complete K\"ahler metric on $X_c\cap X^o$ which we need.
\end{proof}
\section{Harmonic Representation of the derived pushforwards}
Let $f:X\rightarrow Y$ be a proper locally K\"ahler morphism between complex spaces and $X^o\subset X_{\rm reg}$ a dense Zariski open subset. Let $(E,h)$ be a hermitian vector bundle on $X^o$ with a Nakano semi-positive curvature. The aim of this section is to introduce a harmonic representation of $R^qf_\ast S_X(E,h)$. The results of this section are generalizations of Section 4 of  \cite{Takegoshi1995} to $L^2$-Dolbeault complexes on complex spaces.
\subsection{Harmonic forms on weakly 1-complete spaces}
Let $Z$ be a complex space of pure dimension $n$ and $Z^o\subset Z_{\rm reg}$ a dense Zariski open subset. Let $ds^2$ be K\"ahler metric on $Z^o$ which is locally complete and admits a $(\infty,1)$ bounded potential locally on $X$ (Definition \ref{align_various_metric}). Denote $\omega$ to be the associated K\"ahler form.  Let $(E,h)$ be a hermitian vector bundle on $Z^o$. Let $A^{p,q}(Z^o,E)$ (resp. $A^{k}(Z^o,E)$) be the space of $C^\infty$ $E$-valued $(p,q)$ (resp. $k$) forms on $Z^o$ and let $A^{p,q}_{\rm cpt}(Z^o,E)\subset A^{p,q}(Z^o,E)$ be the subspace of forms with compact support in $Z^o$.
Let $\ast:A^{p,q}(Z^o,E)\to A^{n-q,n-p}(Z^o,E)$ be the Hodge star operator relative to $ds^2$ and let $\sharp_E:A^{p,q}(Z^o,E)\to A^{q,p}(Z^o,E^\ast)$ be the anti-isomorphism defined by $h$. Denote by $\langle-,-\rangle_h$ the pointwise inner product on $A^{p,q}(Z^o,E)$. These operators are related by
\begin{align}
\langle\alpha,\beta\rangle_h{\rm vol}_{ds^2}=\alpha\wedge\ast\sharp_E\beta.
\end{align}
Denote 
\begin{align}
(\alpha,\beta)_h:=\int_{Z^o}\langle\alpha,\beta\rangle_h{\rm vol}_{ds^2}
\end{align}
and $\|\alpha\|_h:=\sqrt{(\alpha,\alpha)_h}$.
Let $\nabla=D'+\dbar$ be the Chern connection associated to $h$. Let
$\dbar^\ast_h=-\ast D'\ast$ and $D'^\ast_h=-\ast\dbar\ast$ be the formal adjoints of $\dbar$ and $D'$ respectively.

For two operators $S$ and $T$ acting on $A^r(Z^0,E)$ with degree $a$ and $b$ respectively, we define the graded Lie bracket $[S,T]:=S\circ T-(-1)^{ab}T\circ S$. 

Denote by $L$ the Lefschetz operator with respect to $ds^2$ and by $\Lambda$ the formal adjoint of $L$.
Then we have the following K\"ahler identities (\cite{Wells1980}, Chapter V):
	\begin{align}\label{align_kahler1}
D_h^{'\ast}=-\sqrt{-1}[\dbar,\Lambda],
	\end{align}
		\begin{align}\label{align_kahler2}
	\dbar_h^{\ast}=\sqrt{-1}[D',\Lambda] \quad and
	\end{align}
		\begin{align}\label{align_theta}
	e(\theta)^{\ast}=\sqrt{-1}[e(\bar{\theta}),\Lambda]
	\end{align}
	for $\theta\in A^{1,0}(Z^o,E)$.

 Denote $\Delta_{\dbar}=\dbar\dbar^\ast_h+\dbar^\ast_h\dbar$ and $\Delta_{D'}=D'D'^\ast_h+D'^\ast_hD'$. 
Since $e(\Theta_h)=[\dbar,D']$, by (\ref{align_kahler1}), (\ref{align_kahler2}) and Jacobi's identity 
$$[[S,T],U]+(-1)^{a(b+c)}[[T,U],S]+(-1)^{c(a+b)}[[U,S],T]\equiv0$$
where $a=\deg S$, $b=\deg T$ and $c=\deg U$ respectively, we obtain that the formula
\begin{align}\label{align_Bochner_formula}
\Delta_{\dbar}=\Delta_{D'}+\sqrt{-1}[e(\Theta_h),\Lambda]
\end{align}
holds on $Z^o$ where $e(\alpha)(\beta):=\alpha\wedge\beta$.

Let $\varphi:Z^o\to\bR$ be a $C^\infty$ function and let $h_{\varphi}:=e^{-\varphi}h$. 
Since $D'_{h_{\varphi}}=D'-e(\partial \varphi)$, we obtain the formulae:
\begin{align}\label{align_metric replace}
\dbar_{h_{\varphi}}^{\ast}=\dbar_h^{\ast}+e(\dbar \varphi)^{\ast}
\end{align}
and
\begin{align}
\Delta_{\dbar_{h_{\varphi}}}=\Delta_{D'_{h_{\varphi}}}+\sqrt{-1}[e(\Theta(h)+\partial\dbar \varphi),\Lambda]
\end{align}
where $\Delta_{\dbar_{h_{\varphi}}}:=[\dbar,\dbar_{h_{\varphi}}^{\ast}]$ and $\Delta_{D'_{h_{\varphi}}}:=[D'_{h_{\varphi}}, D'^\ast_{h_{\varphi}}]$.

By (\ref{align_kahler1}), (\ref{align_kahler2}), (\ref{align_theta}) and Jacobi's identity, the Donnelly and Xavier's formula (\cite[(1.9), (1.10)]{Takegoshi1995}) can be stated as follows.
\begin{align}\label{align_DX_formula1}
[\dbar,e(\dbar\varphi)^\ast]+[D'^\ast_h,e(\partial\varphi)]=\sqrt{-1}[e(\ddbar\varphi),\Lambda]
\end{align}
and
\begin{align}\label{align_DX_formula2}
[D',e(\partial\varphi)^\ast]+[\dbar^\ast_h,e(\dbar\varphi)]=\sqrt{-1}[e(\ddbar\varphi),\Lambda)].
\end{align}

Fix a Whitney stratification of $Z$ so that $Z^o$ is the union of open strata. By Sard's theorem there is a subset $\Sigma\subset\bR$ of measure zero so that for every $c\in\bR\backslash\Sigma$ and every stratum $S$ of $Z$, $c$ is a regular value of $\varphi|_S$. Any $c\in\bR\backslash\Sigma$ is called a regular value of $\varphi$. For any regular value $c$ of $\varphi$, $\{\varphi=c\}$ is a piecewise smooth submanifold of $Z$. In particular, $\{\varphi\leq c\}\cap Z^o$ is a submanifold of $Z^o$ with a smooth boundary $\{\varphi=c\}\cap Z^o$. 

Denote $$(\alpha,\beta)_{c,h}=\int_{\{\varphi\leq c\}\cap Z^o}\langle\alpha,\beta\rangle_h{\rm vol}_{ds^2},\quad [\alpha,\beta]_{c,h}:=\int_{\{\varphi=c\}\cap Z^o}\langle\alpha,\beta\rangle_h{\rm vol}_{\{\varphi=c\}}$$
and $\|\alpha\|_{c,h}:=\sqrt{(\alpha,\alpha)_{c,h}}$ for any regular value $c$ of $\varphi$.
By Stokes's theorem we acquire that (\cite[(1.1)]{Takegoshi1995})
\begin{align}\label{align_stokes1}
(\dbar\alpha,\beta)_{c,h}=(\alpha,\dbar^\ast_h\beta)_{c,h}+[\alpha,e(\dbar\varphi)^\ast\beta]_{c,h},
\end{align}
and
\begin{align}\label{align_stokes2}
(D'\alpha,\beta)_{c,h}=(\alpha,D'^\ast_h\beta)_{c,h}+[\alpha,e(\partial\varphi)^\ast\beta]_{c,h}
\end{align}
hold for every $C^\infty$ forms $\alpha$ and $\beta$ on $\{\varphi\leq c\}\cap Z^o$ such that either $\alpha$ or $\beta$ has compact support in $\{\varphi\leq c\}\cap Z^o$.
\begin{prop}\label{prop_keytech}
	Let $(M,\omega_M)$ be a complete K\"ahler manifold of dimension $n$ and let $(E,h)$ be a Nakano semi-positive holomorphic vector bundle on $M$. If $q\geq 1$ and $\alpha\in\sD^{n,q}(M,E;\omega_M,h)\cap {\rm Ker}\Delta_{\dbar}$, then $\alpha$ satisfies the following equations:
	\begin{enumerate}
		\item $\dbar\alpha=0$, $\dbar^\ast_h\alpha=0$, $D'^\ast_h\alpha=0$ and $\langle e(\Theta_h)\Lambda \alpha,\alpha\rangle_h=0$ on $M$. In particular $\dbar(\ast\alpha)=0$, i.e. $\ast\alpha\in \Gamma(M,\Omega^{n-q}_M(E))$.
		\item $e(\dbar\varphi)^\ast\alpha=0$ and $\langle e(\ddbar\varphi)\Lambda \alpha,\alpha\rangle_h=0$ on $M$ for any $C^\infty$ psh function $\varphi$ on $M$ with
		$${\rm sup}_{x\in M}\{|\varphi(x)|+|d\varphi(x)|_{\omega_M}\}<\infty.$$
	\end{enumerate}
\end{prop}
\begin{proof}
	See \cite[Theorem 3.4]{Takegoshi1995}.
\end{proof}
Denote by $\sP_e(Z)$ the set of $C^\infty$ psh functions $\varphi:Z\to(-\infty,c_\ast)$ for some $c_\ast\in(-\infty,\infty]$ such that 
$X_c:=\{z\in Z|\varphi(z)<c\}$ is precompact in $Z$ for every $c<c_\ast$.
$Z$ is called a \textbf{weakly 1-complete space} if $\sP_e(Z)\neq\emptyset$. For every $\varphi\in \sP_e(Z)$, denote
\begin{align}
\sH^{p,q}(Z,E,h,\varphi):=\left\{\alpha\in\sD^{p,q}_{Z,ds^2}(E,h)(Z)\big|\dbar\alpha=\dbar^\ast_h\alpha=0, e(\dbar\varphi)^\ast\alpha=0\right\}.
\end{align}
By the regularity theorem for elliptic operators of second order, every element of $\sH^{p,q}(Z,E,h,\varphi)$ is $C^\infty$ on $Z^o$.
\begin{prop}\label{prop_sH}
	Let $Z$ be a weakly 1-complete space of pure dimension $n$ and $Z^o\subset Z_{\rm reg}$ a dense Zariski open subset. Let $ds^2$ be a K\"ahler metric on $Z^o$ which is locally complete on $Z$ and let $(E,h)$ be a hermitian vector bundle on $Z^o$ with a Nakano semi-positive curvature. Then the following assertions hold.
	\begin{enumerate}
		\item Let $\varphi\in\sP_e(Z)$. Assume that $\alpha\in \sD^{n,q}_{Z,ds^2}(E,h)(Z)$ satisfies $e(\dbar\varphi)^\ast\alpha=0$. Then $\dbar\alpha=\dbar^\ast_h\alpha=0$ if and only if $D_h'^\ast\alpha=\langle \sqrt{-1}e(\Theta_h+\ddbar\varphi)\Lambda\alpha,\alpha\rangle_h=0$. Here $\langle \sqrt{-1}e(\Theta_h+\ddbar\varphi)\Lambda\alpha,\alpha\rangle_h=0$ is equivalent to $\langle \sqrt{-1}(\ddbar\varphi)\Lambda\alpha,\alpha\rangle_h=\langle \sqrt{-1}(\Theta_h)\Lambda\alpha,\alpha\rangle_h=0$.
		\item Let $\varphi,\psi\in \sP_e(Z)$. Then  $ \sH^{n,q}(Z,E,h,\varphi)=\sH^{n,q}(Z,E,h,\psi)$ for every $q\geq 0$.
		\item For every $q\geq0$ and every $\varphi\in \sP_e(Z)$, the Hodge star operator gives a well defined map
		\begin{align}
		\ast:\sH^{n,q}(Z,E,h,\varphi)\to{\rm Ker}\dbar\left(\sD^{n-q,0}_{Z,ds^2}(E,h)(Z)\to \sD^{n-q,1}_{Z,ds^2}(E,h)(Z)\right).
		\end{align}
	\end{enumerate}
\end{prop}
\begin{proof}
	To show (1), we suppose that $e(\dbar\varphi)^\ast\alpha=0$ and $\dbar\alpha=\dbar^\ast_h\alpha=0$. Take any regular value $c$ of $\varphi$. Since $(\{\varphi\leq c\}\cap Z^o,ds^2)$ is complete, the Hopf-Rinow theorem implies that there exists an exhaustive sequence $(K_{\nu})$ of compact sets of $Z^o\cap \{\varphi\leq c\}$ and functions $\theta_{\nu}\in C^{\infty}( \{\varphi\leq c\}\cap Z^o,\bR)$ such that
	$$\theta_{\nu}=1\quad \textrm{on a neighbourhood of $K_{\nu}$,} \quad \textrm{Supp $\theta_{\nu}\subset K^{\circ}_{\nu+1}$},$$
	$$0\leq \theta_{\nu}\leq\theta_{\nu+1}\leq 1\quad \textrm{and}\quad  |d\theta_{\nu}|_g\leq 2^{-\nu},\quad \forall\nu.$$
Then $\theta_{\nu}\alpha$  converges to $\alpha$ under the norm $\|-\|+\|\dbar-\|+\|D'-\|+\|\dbar^\ast_h-\|+\|D'^\ast_h-\|$.
	
	Let $h_\varphi=e^{-\varphi}h$. Since $\dbar\alpha=\dbar^\ast_h\alpha=0$, by (\ref{align_Bochner_formula}) we obtain that 
	\begin{align}
	(D'D'^\ast_{h_\varphi}\alpha,\theta_{\nu}\alpha)_{c,h_\varphi}+(\sqrt{-1}e(\Theta_{h_\varphi})\Lambda \alpha,\theta_{\nu}\alpha)_{c,h_\varphi}=0.
	\end{align}
	By (\ref{align_stokes2}), this is equivalent to
	\begin{align}\label{align_H_1}
	(D'^\ast_{h_\varphi}\alpha,D'^\ast_{h_\varphi}(\theta_{\nu}\alpha))_{c,h_\varphi}+(\sqrt{-1}e(\Theta_{h_\varphi})\Lambda \alpha,\theta_{\nu}\alpha)_{c,h_\varphi}+[e(\partial\varphi)D'^\ast_{h_\varphi}\alpha,\theta_{\nu}\alpha]_{c,h_\varphi}=0.
	\end{align}
	By the assumptions and (\ref{align_DX_formula1}) we have 
	\begin{align}\label{align_H_0}
	e(\partial\varphi)D'^\ast_{h_\varphi}\alpha=\sqrt{-1}e(\ddbar\varphi)\Lambda \alpha.
	\end{align}
	Substituting (\ref{align_H_0}) into (\ref{align_H_1}),  we get that 
	\begin{align}
	(D'^\ast_{h_\varphi}\alpha,D'^\ast_{h_\varphi}(\theta_{\nu}\alpha))_{c,h_\varphi}+(\sqrt{-1}e(\Theta_{h_\varphi})\Lambda \alpha,\theta_{\nu}\alpha)_{c,h_\varphi}+[\sqrt{-1}e(\ddbar\varphi)\Lambda \alpha,\theta_{\nu}\alpha]_{c,h_\varphi}=0.
	\end{align}
	Letting $\nu\to\infty$, we obtain that 
	\begin{align}\label{align_H_2}
	\|D'^\ast_{h_\varphi}\alpha\|^2_{c,h_\varphi}+(\sqrt{-1}e(\Theta_{h_\varphi})\Lambda \alpha,\alpha)_{c,h_\varphi}+[\sqrt{-1}e(\ddbar\varphi)\Lambda \alpha,\alpha]_{c,h_\varphi}=0.
	\end{align}
	
Since $\varphi$ is a psh function and $\sqrt{-1}\Theta_h$ is Nakano semi-positive, $\sqrt{-1}\Theta_{h_\varphi}=\sqrt{-1}\Theta_{h}+\sqrt{-1}\ddbar\varphi$ is also Nakano semi-positive. Then all the three terms in (\ref{align_H_2}) are semi-positive. Hence they are all zero for every regular value $c$ of $\varphi$. This proves the necessity of (1).
	
	To prove the sufficiency of (1), we assume that $D_h'^\ast\alpha=\sqrt{-1}\langle e(\Theta_h+\ddbar\varphi)\Lambda\alpha,\alpha\rangle_h=0$ and $e(\dbar\varphi)^\ast\alpha=0$. By (\ref{align_Bochner_formula}) we have $\Delta_{\dbar}\alpha=0$. By (\ref{align_stokes1}) we obtain
	\begin{align}\label{align_H_3}
	&(\Delta_{\dbar}\alpha,\theta_\nu\alpha)_{c,h}\\\nonumber
	=&(\dbar^\ast_h\alpha,\dbar^\ast_h(\theta_\nu\alpha))_{c,h}+(\dbar\alpha,\dbar(\theta_\nu\alpha))_{c,h}+[\dbar^\ast_h\alpha,e(\dbar\varphi)^\ast(\theta_\nu\alpha)]_{c,h}-[e(\dbar\varphi)^\ast\dbar\alpha,\theta_\nu\alpha]_{c,h}\\\nonumber
	=&0.
	\end{align}
	By (\ref{align_DX_formula1}) we acquire that 
	\begin{align}\label{align_H_4}
	\langle \sqrt{-1}e(\dbar\varphi)\dbar\alpha,\theta_\nu\alpha\rangle_h=\langle \sqrt{-1}e(\ddbar\varphi)\Lambda\alpha,\theta_\nu\alpha\rangle_h=\theta_\nu\sqrt{-1}\langle e(\ddbar\varphi)\Lambda\alpha,\alpha\rangle_h=0.
	\end{align}
	Combining (\ref{align_H_3}), (\ref{align_H_4}) and  $e(\dbar\varphi)^\ast(\theta_\nu\alpha)=\theta_\nu e(\dbar\varphi)^\ast\alpha=0$, we obtain that 
	\begin{align}
	(\dbar^\ast_h\alpha,\dbar^\ast_h(\theta_\nu\alpha))_{c,h}+(\dbar\alpha,\dbar(\theta_\nu\alpha))_{c,h}=0.
	\end{align}
	Taking its limit we know that 
	\begin{align}
	\|\dbar^\ast_h\alpha\|^2_{c,h}+\|\dbar\alpha\|^2_{c,h}=0
	\end{align}
	for every regular value $c$ of $\varphi$. This implies that  $\dbar\alpha=\dbar^\ast_h\alpha=0$.
	
	To prove (2) we set $h_{-\psi}=e^\psi h$. By (\ref{align_DX_formula1}) we obtain that 
\begin{align}\label{thetaalpha}
\dbar e(\dbar\psi)^\ast(\theta_\nu\alpha)+e(\dbar \psi)^\ast \dbar(\theta_{\nu}\alpha) =\sqrt{-1}e(\ddbar\psi)\Lambda(\theta_\nu\alpha)
\end{align}
	if $\alpha\in \sH^{n,q}(Z,E,h,\varphi)$. By (\ref{align_metric replace}) and (\ref{align_stokes1}), we get that 
	\begin{align}
	(\dbar e(\dbar\psi)^\ast(\theta_\nu\alpha),\alpha)_{c,h_{-\psi}}&=(e(\dbar\psi)^\ast(\theta_\nu\alpha),\dbar^\ast_{h_{-\psi}}\alpha)_{c,h_{-\psi}}+[e(\dbar\psi)^\ast(\theta_\nu\alpha),e(\dbar\varphi)^\ast\alpha]_{c,h_{-\psi}}\\\nonumber
	&=(e(\dbar\psi)^\ast(\theta_\nu\alpha),(\dbar^\ast_{h}-e(\dbar\psi)^\ast)\alpha)_{c,h_{-\psi}}\\\nonumber
	&=-(e(\dbar\psi)^\ast(\theta_\nu\alpha),e(\dbar\psi)^\ast\alpha)_{c,h_{-\psi}}.
	\end{align}
		Taking $\nu\to\infty$ on (\ref{thetaalpha}), we know that 
	\begin{align}
	(\sqrt{-1}e(\ddbar\psi)\Lambda(\alpha),\alpha)_{c,h_{-\psi}}+\|e(\dbar\psi)^\ast\alpha\|^2_{c,h_{-\psi}}=0
	\end{align}
	for every regular value $c$ of $\varphi$.
    Since both terms are semi-positive, we show that $e(\dbar\psi)^\ast\alpha=0$. Hence $\alpha\in \sH^{n,q}(Z,E,h,\psi)$.
    
	It remains to show (3). Since $\ast$ is a bounded operator, it suffices to show that $\dbar\ast\alpha=0$ for every $\alpha\in\sH^{n,q}(Z,E,h,\varphi)$. This follows from $-\ast\dbar\ast\alpha=D'^\ast_h\alpha=0$ which is proved in (1).
\end{proof}
By Proposition \ref{prop_sH}-(2), $\sH^{n,q}(Z,E,h,\varphi)$ is independent of the choice of $\varphi\in\sP_e(Z)$. Hence we simply denote $\sH^{n,q}(Z,E,h):=\sH^{n,q}(Z,E,h,\varphi)$ for every $Z$ with $\sP_e(Z)\neq\emptyset$. 

\subsection{Harmonic Representation}
Let us return to the relative setting. Let $f:X\rightarrow Y$ be a proper morphism from a complex space $X$ to an irreducible complex space $Y$. Denote $n:={\rm dim} X$ and $m:={\rm dim} Y$ respectively. Assume that every irreducible component of $X$ is mapped onto $Y$. Let $X^o\subset X_{\rm reg}$ be a dense Zariski open subset and $(E,h)$ a hermitian vector bundle on $X^o$ with a Nakano semi-positive curvature. Assume that there is a K\"ahler metric $ds^2$ on $X^o$ such that
\begin{enumerate}
	\item $ds^2$ admits a $(\infty,1)$ bounded potential locally on $X$;
	\item $ds^2$ is locally complete on $X$ and is locally bounded from below by a hermitian metric.
\end{enumerate}
By Lemma \ref{lem_complete_metric_exists_locally}, such kind of metric exists locally near every fiber of $f$ and globally on $X^o$ when $X$ is a compact K\"ahler space.
By Theorem \ref{thm_main_local1} and Lemma \ref{lem_fine_sheaf}, there is a resolution by fine sheaves
\begin{align}\label{align_resolution_S}
S_X(E,h)\to \sD^{n,\bullet}_{X,ds^2}(E,h).
\end{align}
Denote $X(T):=f^{-1}T$ for every subset $T\subset Y$. Denote by $L$ the Lefschetz operator with respect to $ds^2$ and denote by $\Lambda$ the formal adjoint of $L$.
\begin{prop}\label{prop_restriction_welldefined}
	Notations as above. Let $S\subset Y$ be a Stein open subset. Let $S^o\subset S_{\rm reg}$ and $X(S)^o\subset X(S)_{\rm reg}\cap X(S^o)$ be dense Zariski open subsets so that $f:X(S)^o\to S^o$ is a submersion. 
	Then the Hodge star operator
	\begin{align}\label{align_Hodgestar_welldefined}
	\ast:\sH^{n,q}(X(S),E,h)\to \left\{\alpha\in\sD^{n-q,0}_{X(S),ds^2}(E,h)(X(S))\bigg|\dbar\alpha=0,\alpha|_{X(S)^o}\in\Gamma(X(S)^o,\Omega_{X^o}^{n-m-q}\otimes f^\ast\Omega^m_{S^o})\right\}
	\end{align}
	is well defined and injective for every $q$. As a consequence,
	\begin{enumerate}
		\item $\sH^{n,q}(X(S),E,h)=0$ for every $q>n-m$;
		\item For every open subset $S'\subset S$ such that $\sP_e(S')\neq\emptyset$, the restriction map
		\begin{align}
		\sH^{n,q}(X(S),E,h)\to \sH^{n,q}(X(S'),E,h)
		\end{align}
		is well defined.
	\end{enumerate}
\end{prop}
\begin{proof}
	Let $\alpha\in \sH^{n,q}(X(S),E,h)$. By Proposition \ref{prop_sH}-(1), $D'^\ast_h\alpha=0$, i.e. $\dbar\ast\alpha=0$. It remains to show that $\ast\alpha|_{X(S)^o}\in\Gamma(X(S)^o,\Omega_{X^o}^{n-m-q}\otimes f^\ast\Omega^m_{S^o})$. 
	
	Fix a closed immersion $S\subset\bC^N$ where $z_1,\dots,z_N$ is the coordinate of $\bC^N$. Let $\varphi=\sum_{i=1}^N\|z_i\|^2\in\sP_e(S)$.
	By Proposition \ref{prop_sH}-(1) and (\ref{align_theta}), one has
	\begin{align}
	\sum_{i=1}^N|e(\overline{\partial f^\ast(z_i)})^\ast\alpha|^2_h
	=&\sum_{i=1}^N\langle e(\overline{\partial f^\ast(z_i)})^{\ast}\alpha,\sqrt{-1}e(\partial f^{\ast}(z_i))\Lambda \alpha\rangle\quad(\ref{align_theta})\\\nonumber
	=&\sum_{i=1}^N\langle \sqrt{-1}e(\overline{\partial f^\ast(z_i)})e(\partial f^{\ast}(z_i))\Lambda \alpha,\alpha\rangle\\\nonumber
	=&\langle \sqrt{-1}e(\ddbar f^{\ast}\varphi)\Lambda\alpha,\alpha\rangle_h=0.
	\end{align}	
	Hence $f^\ast(dz_i)\wedge \ast\alpha=0$, $\forall i=1,\dots,N$. This implies that $\alpha|_{X(S)^o}=0$ if $q>n-m$ and $\ast\alpha|_{X(V)\cap X(S)^o}$ can be divided by $f^\ast\theta$ for any open subset $V\subset S^o$ which admits a non-vanishing holomorphic $m$-form $\theta$. Therefore $$\ast\alpha|_{X(S)^o}\in\Gamma(X(S)^o,\Omega_{X^o}^{n-m-q}\otimes f^\ast\Omega^m_{S^o}).$$ This proves that (\ref{align_Hodgestar_welldefined}) is well defined and $\sH^{n,q}(X(S),E,h)=0$ for every $q>n-m$.  (\ref{align_Hodgestar_welldefined}) is injective because  
	\begin{align}\label{align_Last=1}
	L^q\circ\ast=c(n,q){\rm Id},\quad c(n,q)=\sqrt{-1}^{(n-q)(n-q+3)}q!
	\end{align}
 holds	on $(n,q)$-forms (\cite[Theorem 3.16]{Wells1980}).
	
	It remains to show that $\alpha|_{X(S')}\in \sH^{n,q}(X(S'),E,h)$, i.e. $e(\dbar f^\ast\psi)^\ast\alpha|_{X(S')}=0$ for some $\psi\in\sP_e(S')$. By (\ref{align_Last=1}), $\alpha=L^q\beta$ for $\beta=c(n,q)^{-1}\ast\alpha$. Choose an arbitrary $\psi\in\sP_e(S')$. Since $\beta|_{X(S)^o}\in\Gamma(X(S)^o,\Omega_{X^o}^{n-m-q}\otimes f^\ast\Omega^m_{S^o})$, $e(\partial(f^\ast\psi))L^{ k}\beta|_{X(S')\cap X(S)^o}=0$ for every $k=0,\dots,q-1$. Because 
	\begin{align}
	\sqrt{-1}[e(\dbar(f^\ast\psi))^\ast,L]=e(\partial(f^\ast\psi)),\textrm{\quad(\cite[Chapter V, (3.22)]{Wells1980})},
	\end{align}
	we obtain that 
	\begin{align}
	e(\dbar(f^\ast\psi))^\ast L^q\beta|_{X(S')\cap X(S)^o}=L^qe(\dbar(f^\ast\psi))^\ast\beta|_{X(S')\cap X(S)^o}=0.
	\end{align}
	Thus $e(\dbar(f^\ast\psi))^\ast\alpha|_{X(S')}=0$ by continuity. This proves the proposition.
\end{proof}
By Proposition \ref{prop_restriction_welldefined}, the restriction map 
$$\sH^{n,q}(X(V),E,h)\to \sH^{n,q}(X(U),E,h)$$
is well defined for any pair of Stein open subsets $U\subset V\subset Y$. Hence the data
\begin{align*}
U\mapsto\sH^{n,q}(X(U),E,h),\quad U\subset Y \textrm{ is a Stein open subset}
\end{align*}
determines a sheaf $\sH^{n,q}_f(E,h)$ on $Y$ (after a sheafification). 

By (\ref{align_resolution_S}) and Lemma \ref{lem_fine_sheaf}, there is a natural morphism
\begin{align}\label{align_sH_to_Rf}
\sH^{n,\bullet}_f(E,h)\to f_\ast(\sD_{X,ds^2}^{n,\bullet}(E,h))\simeq Rf_\ast(S_X(E,h)).
\end{align}
This induces a canonical morphism
$$\sH^{n,q}_f(E,h)\to R^qf_\ast S_X(E,h)$$
for every $0\leq q\leq n$.
The main result of this section is 
\begin{thm}\label{thm_thm_harmonic_rep}
	$\sH^{n,q}_f(E,h)$ is a sheaf of $\sO_Y$-modules for every $q\geq0$.
	Assume that $S_X(E,h)$ is a coherent sheaf, then the canonical morphism
	$$\sH^{n,q}_f(E,h)\to R^qf_\ast S_X(E,h)$$
	is an isomorphism of $\sO_Y$-modules for every $q\geq0$. Moreover, 
	$$\sH^{n,q}_f(E,h)(U)=\sH^{n,q}(f^{-1}(U),E,h)$$
	for every Stein open subset $U\subset Y$.
\end{thm}
\begin{proof}
	Let $S\subset Y$ be a Stein open subset and $\varphi\in\sP_e(X(S))$.
	For every $\alpha\in\sH^{n,q}(X(S),E,h)$ and every $g\in\sO_Y(S)$, denote  $g':=f^\ast g$. Then Proposition \ref{prop_sH}-(1) implies that 
	$$D'^\ast_h(g'\alpha)=-\ast\dbar\ast(g'\alpha)=g' D'^\ast_h\alpha=0$$ and
	$$\langle \sqrt{-1}e(\Theta_h+\ddbar\varphi)\Lambda(g'\alpha),g'\alpha\rangle_h=\|g'\|^2\langle \sqrt{-1}e(\Theta_h+\ddbar\varphi)\Lambda(\alpha),\alpha\rangle_h=0.$$
	Hence $g'\alpha\in \sH^{n,q}(X(S),E,h,\varphi)$ by Proposition \ref{prop_sH}-(1).
	This shows that $\sH^{n,q}_f(E,h)$ is a sheaf of $\sO_Y$-modules for every $0\leq q\leq n$.
	
	Since $S_X(E,h)$ is a coherent sheaf, to prove the remaining claims, it suffices to show that the natural morphism
	\begin{align}
	\tau^q_U:\sH^{n,q}(X(U),E,h)\to H^q\Gamma(X(U),\sD^{n,\bullet}_{X,ds^2}(E,h))
	\end{align}
	is an isomorphism for every Stein open subset $U\subset Y$ and every $q\geq 0$.  Fix a $C^\infty$ exhausted strictly psh function $\varphi_U$ on $U$. Denote $U_c:=\{\varphi_U<c\}$ and $\varphi:=f^\ast\varphi_U$.
	
	{\bf Claim 1: $\tau^q_U$ is injective.} Assume that $\alpha\in \sH^{n,q}(X(U),E,h)$ and $\alpha=\dbar\beta$ for some $\beta\in \sD^{n,q-1}_{X,ds^2}(E,h)(X(U))$. By (\ref{align_stokes1}), we obtain that 
	$$(\alpha,\alpha)_{c,h}=(\beta,\dbar^\ast_h\alpha)_{c,h}=0$$
	for every regular value $c$ of $\varphi$. Hence $\alpha=0$.
	
	{\bf Claim 2: $\tau^q_U$ is surjective.} Let $\alpha\in \Gamma(X(U),\sD^{n,q}_{X,ds^2}(E,h))$ be $\dbar$-closed.
	
	{\bf Step 1:} In this step we show that for every $c\in\bR$, $\alpha|_{X(U_c)}=u_c+\dbar\beta_c$ for some $u_c\in\sH^{n,q}(X(U_c),E,h)$ and $\beta_c\in \Gamma(X(U_c),\sD^{n,q-1}_{X,ds^2}(E,h))$. 
	
	Fix $0<c_0<c_1$. Denote $\omega_{ds^2}$ to be the K\"ahler form associated to $ds^2$ and let $\omega_{\lambda}:=\omega_{ds^2}+\sqrt{-1}\ddbar\lambda(\varphi-c)$ for some smooth convex function $\lambda:\bR\to\bR_{\geq0}$ such that $\lambda(t)=0$ if $t\leq c_0$, $\lambda(t)>0$, $\lambda'(t)>0$, $\lambda''(t)>0$ if $t>c_0$ and $\int_0^{c_1}\sqrt{\lambda''(t)}dt=+\infty$. Let $h_\lambda:=e^{-\lambda(\varphi)}h$. Then $\omega_\lambda\geq\omega_{ds^2}$ is a complete K\"ahler metric on $X(U_{c+c_1})$ and $(E,h_\lambda)$ is Nakano semi-positive. By choosing $\lambda$ sufficiently large we assume that $\alpha\in L^{n,q}_{(2)}(X(U_{c+c_1}),E;\omega_\lambda,h_\lambda)$. Noting that  $\dbar\alpha=0$, there is a unique decomposition $\alpha=v_c+\gamma_c$ so that $v_c$ is a harmonic form on $U_{c+c_1}$ with respect to $\omega_\lambda$ and $h_\lambda$ while $\gamma_c$ lies in the closure of the range of $\dbar$ in the Hilbert space $L^{n,q}_{(2)}(X(U_{c+c_1}),E;\omega_\lambda,h_\lambda)$. Since $\omega_\lambda|_{U_c}=\omega_{ds^2}|_{U_c}$ and $h_\lambda|_{U_c}=h|_{U_c}$, by Proposition \ref{prop_keytech} and Proposition \ref{prop_rangeclosure_solvedbar} below, we see that $v_c|_{U_c}\in \sH^{n,q}(X(U_c),E,h)$ and $\gamma_c|_{U_c}=\dbar\beta_c$ for some $\beta_c\in \Gamma(X(U_{c}),\sD^{n,q-1}_{X,ds^2}(E,h))$.
	
	{\bf Step 2:} By step 1, there is a decomposition $$\alpha|_{X(U_k)}=u_k+\dbar\beta_k$$ where $u_k\in\sH^{n,q}(X(U_k),E,h)$ and $\beta_k\in \Gamma(X(U_k),\sD^{n,q-1}_{X,ds^2}(E,h))$ for every $k\in\bN$. Since $u_{k+1}$ and $u_k$ are cohomologous on $X(U_k)$, we have $u_{k+1}|_{X(U_k)}=u_k$ by Claim 1. Hence there is a unique $u\in\sH^{n,q}(X(U),E,h)$ such that $u|_{X(U_k)}=u_k$. Hence the cohomology classes $[\alpha]$ and $[u]$ are equal  over every $X(U_k)$, $k\in\bN$. Since $U_k$ is a Stein space for each $k\in\bN\cup\{+\infty\}$ and $S_X(E,h)$ is a coherent sheaf, there is a canonical isomorphism
    $$\Gamma(X(U_k),\sD^{n,q}_{X,ds^2}(E,h))\simeq H^q(X(U_k),S_X(E,h))\simeq \Gamma(U_k,R^qf_\ast(S_X(E,h)))$$
    for every $k\in\bN\cup\{+\infty\}$. By regarding $[\alpha]$ and $[u]$ as sections of $R^qf_\ast(S_X(E,h))$ (which are equal over every $U_k$, $k\in\bN$) we obtain that $[\alpha]=[u]$. This proves the surjectivity of $\tau^q_U$.
\end{proof}	
To complete the proof of Theorem \ref{thm_thm_harmonic_rep}, we go on proving the following proposition.
\begin{prop}\label{prop_rangeclosure_solvedbar}
	Assume that $S_X(E,h)$ is a coherent sheaf on $X$. Let $V\subset X$ be an open subset and denote $V^o:=V\cap X^o$. Let $$\dbar:L^{n,q-1}_{(2)}(V^o,E;ds^2,h)\to L^{n,q}_{(2)}(V^o,E;ds^2,h)$$ be the unbounded operator in the sense of distribution where $q\geq 1$. Suppose that  $\alpha\in\overline{{\rm Im}\dbar}$. Then there is an $(n,q-1)$-form $\beta\in L^{n,q-1}_{V,ds^2}(E,h)(V)$ such that $\alpha=\dbar\beta$.
\end{prop}
\begin{proof}
To prove the proposition, we need the following lemma which is a modified version of Theorem \ref{thm_Hormander_incomplete} and we leave its proof to the end of this section.
\begin{lem}\label{lem_estimate}
	Let $Y$ be a complex manifold of dimension $n$ which admits a complete K\"ahler metric. Let $\omega=\sqrt{-1}\ddbar\varphi$ be a K\"ahler metric of $Y$ with $\sup|\varphi|<\infty$. Let $(E,h)$ be a Nakano semi-positive hermitian vector bundle on $Y$. Then for every $q>0$ and every $\alpha\in L^{n,q}_{(2)}(Y,E;\omega,h)$ such that $\dbar\alpha=0$, there is $\beta\in L^{n,q-1}_{(2)}(Y,E;\omega,h)$ such that $\dbar\beta=\alpha$ and $\|\beta\|^2_{\omega,h}\leq C\|\alpha\|^2_{\omega,h}$. Here  $C$ is a constant depending on $q$ and $\varphi$, but not depending on $\alpha$.
\end{lem}
\begin{proof}
	Let $h':=he^{-\varphi}$, then $h'\sim h$. Since $(E,h)$ is Nakano semi-positive, we know that 
	$$\sqrt{-1}\Theta_{h'}(E)=\sqrt{-1}\partial\dbar \varphi\otimes {\rm Id}_E+\sqrt{-1}\Theta_h(E)\geq \omega\otimes {\rm Id}_E.$$
	By Theorem \ref{thm_Hormander_incomplete}, there exists $\beta\in L_{(2)}^{n,q-1}(Y,E;\omega,he^{-\varphi})$ such that $\dbar \beta=\alpha$ and 
	$\|\beta\|_{\omega,he^{-\varphi}}^2\leq \frac{1}{q}\|\alpha\|_{\omega,h^{-\varphi}}$. Since $\varphi$ is a bounded function, we have $\|\beta\|^2_{\omega,h}\leq C\|\alpha\|^2_{\omega,h}$ where $C$ is a constant depending on $q$ and $\varphi$ and not depending on $\alpha$. The lemma is proved.
\end{proof}
	Let us return to the proof of the proposition. First we take a locally finite Stein open cover $\{V_j\}$ of $V$. Denote $V_{j_0,j_1,\dots,j_p}:=V_{j_1}\cap \cdots \cap V_{j_p}$. By Lemma \ref{lem_complete_metric_exists_locally}, after a possible refinement we assume that there is a complete K\"ahler metric on $V_{j_0,j_1,\dots,j_p}^o:=V_{j_0,j_1,\dots,j_p}\cap X^o$ which has a bounded potential for every $j_0,j_1,\dots,j_p$.  Since $\alpha\in \overline{{\rm Im}\dbar}\subset L^{n,q}_{(2)}(V^o,E;ds^2,h) $ and $\rm Ker \dbar$ is closed in $L^{n,q}_{(2)}(V^o,E;ds^2,h)$, we know that $\dbar \alpha=0$. 
	
	Since $S_X(E,h)$ is a coherent sheaf on $X$, by Theorem \ref{thm_main_local1} and Lemma \ref{lem_fine_sheaf} there are isomorphisms of cohomologies
	\begin{align}\label{align_Cech_vs_L2}
	H^k(\{V_j\},S_X(E,h))\simeq H^k(V,S_X(E,h))\simeq H^k(\Gamma(V,\sD^\bullet_{X,E;ds^2,h})),
	\end{align}
	where $H^k(\{V_j\},S_X(E,h))$ is the \v{C}ech cohomology with respect to the covering $\{V_j\}$.
	
	Let us recall the corresponding \v{C}ech cocycle of $\alpha$.
	Let $\alpha_j=\alpha|_{V_j^o}$. By Lemma \ref{lem_estimate}, there exists $b_j\in L_{(2)}^{n,q-1}(V_j^o,E;ds^2,h)$ such that $\alpha_j=\dbar b_j$ on $V_j^o$ for each $j$.
	
	Suppose that $\{\alpha_{j_0,\dots,j_r}\}$ and $\{b_{j_0,\dots,j_r}\}$ are determined in the way as:
	\begin{align}\label{cech}
&b_{j_0,\dots,j_r}\in L_{(2)}^{n,q-r-1}(V_{j_0,j_1,\dots,j_p}^o,E;ds^2,h),\quad
\alpha_{j_0,\dots,j_r}\in L_{(2)}^{n,q-r}(V_{j_0,j_1,\dots,j_p}^o,E;ds^2,h),\\\nonumber
&\alpha_{j_0,\dots,j_r}=\dbar b_{j_0,\dots,j_r} \quad{\rm on}\quad V_{j_0,\dots,j_r}^o \quad {\rm and} \quad (\delta \alpha)_{j_0,\dots,j_{r+1}}=0.
	\end{align}
	
Set $$\alpha_{j_0,\dots,j_{r+1}}:=(\delta  b)_{j_0,\dots,j_{r+1}},$$
which is a $\dbar$-closed form. It follows from Lemma \ref{lem_estimate} that the same statements in (\ref{cech}) also hold for $\alpha_{j_0,\dots,j_{r+1}}$. Repeating the above steps, we can obtain a $q$-cocycle $\{\alpha_{j_0,\dots,j_{q}}\}\in Z^q(\{V_j\},S_X(E,h))$ which corresponds to $\alpha$ by (\ref{align_Cech_vs_L2}). 
	
	By the hypothesis there exits a sequence $\{\gamma_m\}\in L_{(2)}^{n,q-1}(V^o,E;ds^2,h)$ such that $\|\alpha-\dbar\gamma_m\|\rightarrow 0$ as $m\rightarrow \infty$. Let $\nu_m:=\alpha-\dbar \gamma_m$. Since $\dbar \nu_m=0$, by Lemma \ref{lem_estimate} there exists $\{\mu_{j,m}\}\in L_{(2)}^{n,q-1}(V_j^o,E;ds^2,h)$ such that $\nu_m=\dbar \mu_{j,m}$ and $\|\mu_{j,m}\|\leq C_j\|\nu_m\|$ for $C_j$ independent of $m$ but depending on $j$.
	Set $\rho_{j,m}:=b_j-\mu_{j,m}-\gamma_m|_{V_j^o}$. Then we have $\dbar \rho_{j,m}=0$, $\dbar(\alpha_{ij}-\delta\{\rho_{j,m}\})=0$ and $\|\alpha_{ij}-\delta\{\rho_{j,m}\}\|\leq C_{ij}\|\nu_m\|$.
	
	Suppose that $\{\rho_{j_0,\dots,j_{r},m}\}$ are already determined as:
\begin{align}\label{cech2}
&\rho_{j_0,\dots,j_{r},m}\in L_{(2)}^{n,q-r-1}(V_{j_0,\dots,j_r}^o,E;ds^2,h), \dbar(\alpha_{j_0,\dots,j_{r+1}}-(\delta\rho)_{j_0,\dots.,j_{r+1},m})=0,\\\nonumber
&\dbar \rho_{j_0,\dots,j_{r},m}=0,\|\alpha_{j_0,\dots,j_{r+1}}-(\delta\rho)_{j_0,\dots.,j_{r+1},m}\|\leq C_{j_0,\dots,j_{r+1}}\|\nu_m\|.
\end{align}
	We construct $\{\rho_{j_0,\dots,j_{r+1},m}\}$ as follows. (\ref{cech2}) and Lemma \ref{lem_estimate} imply that there exists $\gamma_{j_0,\dots,j_r,m}\in L_{(2)}^{n,q-r-2}(V_{j_0,\dots,j_r}^o,E;ds^2,h)$ and $\mu_{j_0,\dots,j_{r+1},m}\in L_{(2)}^{n,q-r-2}(V_{j_0,\dots,j_{r+1}}^o,E;ds^2,h)$ such that $\rho_{j_0,\dots,j_{r},m}=\dbar \gamma_{j_0,\dots,j_r,m}$, $\alpha_{j_0,\dots,j_{r+1}}-(\delta \rho_m)_{j_0,\dots,j_{r+1}}=\dbar\mu_{j_0,\dots,j_{r+1},m}$ and $\|\mu_{j_0,\dots,j_{r+1},m}\|\leq C_{j_0,\dots,j_{r+1}}\|\nu_{m}\|$. Then we set $$\rho_{j_0,\dots,j_{r+1},m}=b_{j_0,\dots,j_{r+1}}-\mu_{j_0,\dots,j_{r+1},m}-(\delta \gamma)_{j_0,\dots,j_{r+1},m},$$
	which satisfies the statements in (\ref{cech2}). Repeating the above steps, we obtain a $q-1$ cochain $\rho_m=\{\rho_{j_0,\dots,j_{q-1},m}\}\in C^{q-1}(\{V_j\},S_X(E,h))$ such that $\|\alpha-\delta \rho_m\|\rightarrow 0$ as $m\rightarrow \infty$. By \cite[Theorem 2.2.3]{Hormander1990}, $\delta \rho_m$ tends to $\alpha$ as $m\rightarrow \infty$ uniformly on compact subsets of $X$.  Since $S_X(E,h)$ is coherent, $\delta C^{q-1}(\{V_j\},S_X(E,h))$ is a closed subspace of $C^{q}(\{V_j\},S_X(E,h))$ with respect to the topology of uniform convergence on compact subsets. Hence there exists $\rho\in C^{q-1}(\{V_j\},S_X(E,h))$ such that $\delta \rho=\alpha$. By (\ref{align_Cech_vs_L2}), $[\alpha]=0\in H^{q}(\Gamma(V,\sD^\bullet_{X,E;ds^2,h}))$. Thus we prove the proposition.	
\end{proof}
\section{An abstract Koll\'ar package}\label{section_abstract_Kollar_package}
\begin{defn}{\cite[Definition 6.1]{Takegoshi1995}}\label{defn_loc_Kahler}
	A morphism $f:X\rightarrow Y$ between complex spaces is locally K\"ahler if $f^{-1}U$ is a K\"ahler space for any relatively compact open subset $U\subset Y$.
\end{defn}
Throughout this section, $f:X\rightarrow Y$ is a proper locally K\"ahler morphism from a complex space $X$ to an irreducible complex space $Y$. Assume that every irreducible component of $X$ is mapped onto $Y$. $X^o\subset X_{\rm reg}$ is a dense Zariski open subset and $(E,h)$ is a hermitian vector bundle on $X^o$ with a Nakano semi-positive curvature. In this section we establish the Koll\'ar package of $S_X(E,h)$ under the coherence assumptions.
\begin{thm}[Torsion Freeness]
	 Assume that $S_X(E,h)$ is a coherent sheaf on $X$, then $R^qf_\ast S_X(E,h)$ is torsion free for every $q\geq0$ and vanishes if $q>\dim X-\dim Y$.
\end{thm}
\begin{proof}
	Since the problem is local, we assume that $Y$ is Stein and there is a K\"ahler metric on $X^o$ which is locally complete and locally bounded from below by a hermitian metric and which admits a $(\infty,1)$ potential locally on $X$ (Lemma \ref{lem_complete_metric_exists_locally}). Define the sheaf $\Omega^{p}_{X,ds^2,(2)}(E,h)$ as 
	$$\Omega^{p}_{X,ds^2,(2)}(E,h)(U)=\left\{\alpha\in\sD^{p,0}_{X,ds^2}(E,h)(U)\bigg|\dbar\alpha=0\right\}$$
	for every open subset $U\subset X$.
	
	By Proposition \ref{prop_sH}-(3), the Hodge star operator induces a well defined map
	$$\ast:\sH^{n,q}_f(E,h)\to f_\ast \Omega^{n-q}_{X,ds^2,(2)}(E,h).$$
	Since the Lefschetz operator $L$ with respect to $ds^2$ is bounded and $[L,\dbar]=0$, taking the finesss of $\sD^{n,\bullet}_{X,ds^2}(E,h)$ (Lemma \ref{lem_fine_sheaf}) into account we get a well defined map
	$$L^q:f_\ast \Omega^{n-q}_{X,ds^2,(2)}(E,h)\to R^qf_\ast(\sD^{n,\bullet}_{X,ds^2}(E,h)).$$
	By Theorem \ref{thm_main_local1} and Theorem \ref{thm_thm_harmonic_rep}, we get the morphisms
	$$R^qf_\ast S_X(E,h)\stackrel{\ast}{\to}f_\ast \Omega^{n-q}_{X,ds^2,(2)}(E,h)\stackrel{L^q}{\to}R^qf_\ast S_X(E,h)$$
	such that
	\begin{align}
	L^q\circ\ast=c(n,q){\rm Id},\quad c(n,q)=\sqrt{-1}^{(n-q)(n-q+3)}q!\quad (\ref{align_Last=1}).
	\end{align}
	This proves the first claim since $f_\ast \Omega^{n-q}_{X,ds^2,(2)}(E,h)$ is torsion free for every $q\geq 0$. The vanishing result follows from  Proposition \ref{prop_restriction_welldefined}-(1).
\end{proof}
\begin{thm}[Injectivity Theorem]\label{thm_injectivity_formal}
	Assume that $S_X(E,h)$ is a coherent sheaf on $X$. If $L$ is a semi-positive holomorphic line bundle on $X$ so that $L^l$ admits a nonzero holomorphic global section $s$ for some $l>0$, then the canonical morphism
	$$R^qf_\ast(\otimes s):R^qf_\ast(S_X(E,h)\otimes L^{\otimes k})\to R^qf_\ast(S_X(E,h)\otimes L^{\otimes k+l})$$
	is injective for every $q\geq0$ and every $k\geq1$.
\end{thm}
\begin{proof}
	Since the problem is local, we assume that $Y$ is Stein and there is a K\"ahler metric on $X^o$ which is locally complete and locally bounded from below by a hermitian metric and which admits a $(\infty,1)$ potential locally on $X$ (Lemma \ref{lem_complete_metric_exists_locally}). Let $h_L$ be the hermitian metric on $L$ with a semi-positive curvature, then we have
	$$S_X(E,h)\otimes L^{\otimes k}\simeq S_X(E\otimes L^{\otimes k},h\otimes h_L^k),\quad k\geq 1$$
	by Lemma \ref{lem_Kernel}.
	By Theorem \ref{thm_thm_harmonic_rep}, it is therefore sufficient to show that the canonical map
	\begin{align}
	\otimes s:\sH^{n,q}(X(U),E\otimes L^{\otimes k},h\otimes h_L^k)\to \sH^{n,q}(X(U),E\otimes L^{\otimes k+l},h\otimes h_L^{k+l})
	\end{align}
	is well defined for every Stein open subset $U$ of $Y$. 
	Let $\alpha\in\sH^{n,q}(X,E\otimes L^{\otimes k},h\otimes h_L^k,\varphi)$ for some $\varphi\in\sP_e(X)$ ($\sP_e(X)\neq \emptyset$ since $Y$ is Stein and $f$ is proper). It follows from Proposition \ref{prop_sH}-(1) that  
	$$D'^\ast_h(\alpha\otimes s)=-\ast\dbar\ast(\alpha\otimes s)=D'^\ast_h\alpha\otimes s=0$$ and
	\begin{align*}
	0&\leq\langle \sqrt{-1}e(\Theta_{h\otimes h_L^{k+l}}(E\otimes L^{k+l})+\ddbar\varphi)\Lambda(\alpha\otimes s),\alpha\otimes s\rangle_{h\otimes h_L^{k+l}}\\\nonumber
	&\leq\frac{k+l}{k}|s|_{h_L^l}^2\langle \sqrt{-1}e(\Theta_{h\otimes h_L^k}(E\otimes L^k)+\ddbar\varphi)\Lambda(\alpha),\alpha\rangle_{h\otimes h_L^k}=0.
	\end{align*}
	Hence $\alpha\otimes s\in \sH^{n,q}(X,E\otimes L^{\otimes k+l},h\otimes h_L^{k+l},\varphi)$ by Proposition \ref{prop_sH}-(1). Hence $\otimes s$ is well defined and injective. The proof is finished.
\end{proof}
\begin{thm}[Decomposition]\label{thm_decomposition_formal}
	Assume that there exists a K\"ahler metric $ds^2$ on $X^o$ such that
	\begin{enumerate}
		\item $ds^2$ admits a $(\infty,1)$ bounded potential locally on $X$;
		\item $ds^2$ is locally complete on $X$ and is locally bounded from below by a hermitian metric.
	\end{enumerate}
    Assume that $S_X(E,h)$ is a coherent sheaf on $X$. Then $Rf_\ast S_X(E,h)$ splits in $D(Y)$, i.e. 
	$$Rf_\ast S_X(E,h)\simeq \bigoplus_{q} R^qf_\ast S_X(E,h)[-q]\in D(Y).$$
	As a consequence, the spectral sequence
	\begin{align}\label{align_E2_degeneration}
	E^{pq}_2:=H^p(Y,R^qf_\ast S_X(E,h))\Rightarrow H^{p+q}(X,S_X(E,h))
	\end{align}
	degenerates at the $E_2$ page if $Y$ is compact.
\end{thm}
	\begin{rmk}
		By Lemma \ref{lem_complete_metric_exists_locally}, such kind of metric exists if $X$ is compact.
	\end{rmk}
\begin{proof}
	Let 
	\begin{align}
	\alpha:\bigoplus_{q} \sH^{n,q}_f(E,h)[-q]\to f_\ast\sD^\bullet_{X,E;ds^2,h}\simeq Rf_\ast\sD^\bullet_{X,E;ds^2,h}
	\end{align}
	be the inclusion map. 
	By Theorem \ref{thm_thm_harmonic_rep}, $\alpha$ is a quasi-isomorphism under the hypothesis that $S_X(E,h)$ is coherent. Hence by Theorem \ref{thm_main_local1} and Theorem \ref{thm_thm_harmonic_rep} we obtain that 
	$$Rf_\ast S_X(E,h)\simeq Rf_\ast\sD^\bullet_{X,E;ds^2,h}\simeq\bigoplus_{q} \sH^{n,q}_f(E,h)[-q]\simeq\bigoplus_{q} R^qf_\ast S_X(E,h)[-q]$$
	in $D(Y)$. The degeneration of the spectral sequence follows from standard arguments.
\end{proof}
\begin{thm}[Vanishing Theorem]
	Assume that $S_X(E,h)$ is a coherent sheaf on $X$.
	If $Y$ is a projective algebraic variety and $L$ is an ample line bundle on $Y$, then
	$$H^q(Y,R^pf_\ast S_X(E,h)\otimes L)=0,\quad \forall q>0,\forall p\geq0.$$
\end{thm}
\begin{proof}
	Since $S_X(E,h)$ is coherent, so is $R^pf_\ast S_X(E,h)$. Then there is $k$ large enough so that $H^0(Y,L^{\otimes k})\neq 0$ and 
	\begin{align}\label{align_Serre_vanishing}
	H^q(Y,R^pf_\ast S_X(E,h)\otimes L^{k+1})=0,\quad \forall q>0,\forall p\geq0.
	\end{align}
	By Lemma \ref{lem_Kernel}, we acquire that 
	\begin{align}\label{align_vanishing_1}
	S_X(E,h)\otimes f^\ast L^{\otimes l}\simeq S_X(E\otimes L^{\otimes l},h\otimes f^\ast h_L^l),\quad \forall l>0
	\end{align}
	where $h_L$ is a hermitian metric on $L$ with positive curvature.
	By Theorem \ref{thm_injectivity_formal}, the canonical map
	$$\otimes f^\ast s:H^q(X,S_X(E\otimes L,h\otimes f^\ast h_L))\to H^q(X,S_X(E\otimes L^{\otimes k+1},h\otimes f^\ast h_L^{k+1})),\quad \forall q\geq0$$
	is injective. By Theorem \ref{thm_decomposition_formal}, we know that the canonical map
	$$\otimes s:H^q(Y,R^pf_\ast S_X(E\otimes L,h\otimes f^\ast h_L))\to H^q(Y,R^pf_\ast S_X(E\otimes L^{\otimes k+1},h\otimes f^\ast h_L^{k+1})),\quad \forall p,q\geq0$$
	is injective. Combining this with (\ref{align_Serre_vanishing}) and (\ref{align_vanishing_1}), we prove the theorem.
\end{proof}
\section{Non-abelian Hodge theory and Koll\'ar package}
\subsection{Harmonic bundle and variation of Hodge structure}
The notion of harmonic bundle is used by Simpson \cite{Simpson1992} to establish a correspondence between local systems and semistable higgs bundles with vanishing Chern classes over a compact K\"ahler manifold. A typical example of harmonic bundles comes from a polarized variation of Hodge structure (loc. cit.). A harmonic bundle produces a $\lambda$-connection structure which gives a $\sD$-module on $\lambda=1$ and a higgs bundle on $\lambda=0$. This is the main subject of non-abelian Hodge theory. We only review the necessary knowledge of this topic that is used in the present paper. Readers may consult \cite{Simpson1988,Simpson1990,Simpson1992,Sabbah2005,Mochizuki20072,Mochizuki20071} for more details.

Let $(M,ds^2)$ be a K\"ahler manifold and $(\cV,\nabla)$ a holomorphic vector bundle with a flat connection on $M$. Let $h$ be a hermitian metric on $\cV$ which is not necessarily compatible with $\nabla$. 

Let $\nabla=\nabla^{1,0}+\nabla^{0,1}$ be the bi-degree decomposition. There are unique operators
$\delta'_h$ and $\delta''_h$ so that $\nabla^{1,0}+\delta''_h$ and $\nabla^{0,1}+\delta'_h$ are connections compatible with $h$. Denote $\nabla^c_h=\delta''_h-\delta'_h$ and $\Theta_h(\nabla)=\nabla\nabla^c_h+\nabla^c_h\nabla$.
\begin{defn}
	$(\cV,\nabla,h)$ is called a harmonic bundle if $\Theta_h(\nabla)=0$. In this case, $h$ is called a harmonic metric.
\end{defn}
Let $(\cV,\nabla,h)$ be a harmonic bundle. Denote $\theta=\frac{1}{2}(\nabla^{1,0}-\delta'_h)$ and $\dbar=\frac{1}{2}(\nabla^{0,1}+\delta''_h)$. Then $\dbar^2=0$. Denote by $\widetilde{\cV}$ the underlying complex $C^\infty$-vector bundle of $\cV$, then $H=(\widetilde{\cV},\dbar)$ is a holomorphic vector bundle and $\theta$ is a higgs field on $H$ (i.e. $\theta$ is $\sO_M$-linear and $\theta^2=0$). 

Let $\nabla_h$ be the Chern connection on $H$ with respect to $h$ and let  $\Theta_h(H)=\nabla_h^2$ be its curvature form. Then we have the self-dual equation
\begin{align}\label{align_self-dual equation}
\Theta_h(H)+\theta\wedge\overline{\theta}+\overline{\theta}\wedge\theta=0
\end{align}
where $\overline{\theta}=\frac{1}{2}(\nabla^{0,1}-\delta''_h)$ is the adjoint of $\theta$ with respect to the metric $h$.

Conversely, let $(H,\theta,h)$ be a hermitian higgs bundle . Let $\overline{\theta}$ be the adjoint of $\theta$ and let $\partial$ be the unique $(1,0)$-connection such that $\partial+\dbar$ is compatible with $h$. $h$ is called harmonic if $\Theta_h(\theta):=(\partial+\dbar+\theta+\overline{\theta})^2=0$. There is a $1:1$ correspondence (c.f. \cite{Simpson1988,Simpson1992})
\begin{align}\label{align_formal_simpson_corr}
\left\{(\cV,\nabla,h):\Theta_h(\nabla)=0\right\}\stackrel{1:1}{\leftrightarrow}\left\{(H,\theta,h):\Theta_h(\theta)=0\right\}.
\end{align}
Therefore by a harmonic bundle we mean an object on either side of (\ref{align_formal_simpson_corr}).

For the purpose of the present paper, we are interested in tame harmonic bundles in the sense of Simpson \cite{Simpson1990} and Mochizuki \cite{Mochizuki20072,Mochizuki20071}.
\begin{defn}\label{defn_0tame_harmonic_bundle}
	Let $X$ be a complex space and $X^o\subset X_{\rm reg}$ a dense Zariski open subset. A harmonic bundle $(\cV,\nabla,h)$ on $X^o$ is called a tame harmonic bundle on $(X,X^o)$ if, locally at every point $x\in X$, there exists some $c>0$ such that
	\begin{align}\label{align_defn_0tame}
	(\sum_{i=1}^r\|f_i\|^2)^{c}\lesssim|v|_h\lesssim (\sum_{i=1}^r\|f_i\|^2)^{-c}
	\end{align}
holds for every (multivalued) flat section $v$. Here $\{f_1,\dots,f_r\}$ is a local generator of the ideal sheaf defining $X\backslash X^o\subset X$ and $c>0$ is a constant independent of $v$. 
\end{defn}
The tameness is independent of the choice of the local generator $\{f_1,\dots,f_r\}$. Let $\pi:\widetilde{X}\to X$ be a desingularization of $X\backslash X^o\subset X$ so that the preimage of $X\backslash X^o$ is supported on a simple normal crossing divisor $E$. Then (\ref{align_defn_0tame}) is equivalent to the estimate
\begin{align}\label{align_defn_0tame2}
\|z_1\cdots z_r\|^{c'}\lesssim|v|_h\lesssim \|z_1\cdots z_r\|^{-c'},\quad\textrm{ for some }c'>0
\end{align}
where $(z_1,\dots,z_n)$ is an arbitrary holomorphic local coordinate of $\widetilde{X}$ such that $E=\{z_1\cdots z_r=0\}$.
\begin{rmk}\label{rmk_tame_alterdef}
	There is an equivalent definition of tameness, using the corresponding higgs bundle. By desingularization, we assume that $X$ is smooth and $X\backslash X^o\subset X$ is a normal crossing divisor. Then the higgs field $\theta$ of the higgs bundle $(H,\theta)$ associated to $(\cV,\nabla,h)$ can be described as:
	\begin{align}\label{align_tame_higgs}
	\theta=\sum_{i=1}^rf_i\frac{dz_i}{z_i}+\sum_{j=r+1}^ng_jdz_j
	\end{align}
	under a holomorphic coordinate $U\subset X$ such that $U\cap(X\backslash X^o)=\{z_1\cdots z_r=0\}$.
	The harmonic bundle $(\cV,\nabla,h)$ is tame if and only if the coefficients of the characteristic polynomials $\det(t-f_i)$, $i=1,\dots,r$ and $\det(t-g_j)$, $j=r+1,\dots,n$ can be extended
	to the holomorphic functions on $U$. Readers may see \cite{Mochizuki2002} for more details.
\end{rmk}
A typical type of tame harmonic bundles is the variation of Hodge structure.
\begin{defn}{\cite[\S 8]{Simpson1988}}\label{defn_CVHS}
	Let $X$ be a complex space and $X^o\subset X_{\rm reg}$ a dense Zariski open subset. Denote by $\sA^0_{X^o}$ the sheaf of $C^\infty$ functions on $X^o$.
	A polarized complex variation of Hodge structure on $X^o$ of weight $k$ is a harmonic bundle $(\cV,\nabla,h)$ on $X^o$ together with a decomposition $\cV\otimes_{\sO_X}\sA^0_{X^o}=\bigoplus_{p+q=k}\cV^{p,q}$ of $C^\infty$-bundles such that
	\begin{enumerate}
		\item The Griffiths transversality condition 
		\begin{align}\label{align_Griffiths transversality}
		\nabla(\cV^{p,q})\subset \sA^{0,1}(\cV^{p+1,q-1})\oplus \sA^1(\cV^{p,q})\oplus \sA^{1,0}(V^{p-1,q+1})
		\end{align}
		holds for every $p$ and $q$. Here $\sA^{p,q}(\cV^{i,j})$ (resp. $\sA^{k}(\cV^{i,j})$) denotes the sheaf of $C^\infty$ $(p,q)$-forms (resp. $k$-forms) with values in $\cV^{i,j}$.
		\item The hermitian form $Q$ which equals $(-1)^{p}h$ on $\bV^{p,q}$ is parallel with respect to $\nabla$.
	\end{enumerate}
	Denote $S(\bV):=\cF^{p_{\rm max}}$ where $p_{\rm max}=\max\{p|\cF^{p}\neq0\}$. 
\end{defn}
The following proposition is known to experts. We recall the proof in sketch for the convenience of readers.
\begin{prop}\label{prop_VHS_tame}
	If a harmonic bundle $(\cV,\nabla,h)$ admits a complex variation of Hodge structure, then $(\cV,\nabla,h)$ is tame.
\end{prop}
\begin{proof}
	Let $\bV:=(\cV,\nabla,\{\cV^{p,q}\},h_\bV)$ be a complex variation of Hodge structure of weight $k$. Take the decomposition 
	$$\nabla=\overline{\theta}+\partial+\dbar+\theta$$
	according to (\ref{align_Griffiths transversality}).
	By \cite[\S 8]{Simpson1988}, the corresponding higgs bundle of $(\cV,\nabla,h_\bV)$ is $(H={\rm Ker}\dbar,\theta,h_\bV)$. There is moreover an orthogonal decomposition of holomorphic subbundles
	$H=\oplus_{p+q=k}H^{p,q}$ where $H^{p,q}=H\cap\cV^{p,q}$ and 
	$$\theta(H^{p,q})\subset H^{p-1,q+1}\otimes \Omega_{X^o}.$$
	
	Thus the higgs field is nilpotent. Therefore $f_i$, $i=1,\dots, r$ and $g_j$, $j=r+1,\dots, n$ in the local expression (\ref{align_tame_higgs}) are nilpotent matrix. Hence $\det(t-f_i)=t^n$ and $\det(t-g_j)=t^n$. So the harmonic bundle is tame due to the alternative definition of tameness (Remark \ref{rmk_tame_alterdef}).
\end{proof}
\iffalse
When $X^o=X$ is a compact K\"ahler manifold, the class of complex variation of Hodge structure on $X^o$ is exactly the fixed points of the $\bC^\ast$ action on the moduli space of harmonic bundles (\cite[Lemma 4.1]{Simpson1992}). Since this moduli space is homeomorphic to the moduli space of semisimple representation of $\pi_1(X^o)$, complex variation of Hodge structure could be constructed by deforming representation of  $\pi_1(X^o)$ to the fixed point. For the case when $X^o\neq X$ see \cite[Proposition 10.3]{Mochizuki2006}.
\fi
\subsection{Harmonic bundle and Koll\'ar package}\label{section_CVHS_Kollar}

\begin{prop}\label{prop_CVHS_Ssheaf}
Let $X$ be a complex space and $X^o\subset X_{\rm reg}$ a dense Zariski open subset. Let $(H,\theta,h)$ be a tame harmonic bundle on $X^o$. Let $E\subset H$ be a holomorphic subbundle with vanishing second fundamental form. Assume that $\overline{\theta}(E)=0$. Then $(E,h)$ is a tame hermitian vector bundle with a Nakano semi-positive curvature. As a consequence, $S_X(E,h)$ is a coherent sheaf on $X$.
\end{prop}
\begin{proof}
	To prove the tameness (Definition \ref{defn_tame_hermitian_bundle}), we construct $Q$ by Mochizuki's prolongation construction. Since the problem is local, we assume that there is a desingularization  $\pi:\widetilde{X}\to X$ such that $\pi$ is biholomorphic over $X^o$ and $D:=\pi^{-1}(X\backslash X^o)$ is a simple normal crossing divisor. By abuse of notations we identify $X^o$ and $\pi^{-1}X^o$. Since  $(H,\theta,h)$ is tame, by \cite[Theroem 8.58]{Mochizuki20072} there is a logarithmic higgs bundle $(\widetilde{H},\widetilde{\theta})$:
	\begin{align*}
	\widetilde{\theta}: \widetilde{H}\to\Omega_{\widetilde{X}}(\log D)\otimes\widetilde{H},
	\end{align*}
	such that $(\widetilde{H},\widetilde{\theta})|_{X^o}$ is holomorphically equivalent to $(H,\theta)$. Let $D=\cup_{i=1}^k D_i$ be the irreducible decomposition and let $(z_1,\dots,z_n)$ be a local coordinate such that $D_i=\{z_i=0\}$, $i=1,\dots,k$. By \cite[Part 3, Chapter 13]{Mochizuki20072}, the tameness of $(H,\theta,h)$ forces a norm estimate 
	\begin{align}\label{align_normest_tame}
	|z_1\cdots z_k|^{2b}\lesssim |s|_h
	\end{align}
	for any local holomorphic section $s$ of $\widetilde{H}$ and a constant $b>0$ which is independent of $s$.
	
	Let $Q:=\widetilde{H}$, then $E$ is a holomorphic subbundle of $Q|_{X^o}$. By (\ref{align_normest_tame}), we see that $(E,h)$ is tame.
	
	To see that $(E,h)$ is Nakano semi-positive, we take the decomposition 
	$$\nabla=\overline{\theta}+\partial+\dbar+\theta$$
	according to (\ref{align_Griffiths transversality}).
	Since $E\oplus E^\bot$ is an orthogonal holomorphic direct sum of $H$ and $\overline{\theta}(E)=0$, we get from (\ref{align_self-dual equation}) that
	$$\sqrt{-1}\Theta_h(E)=\sqrt{-1}\theta\wedge\overline{\theta}\geq 0.$$
	Hence $(E,h)$ is Nakano semi-positive. By Proposition \ref{prop_S_coherent}, $S_X(E,h)$ is a coherent sheaf on $X$.
\end{proof}
Let $\bV:=(\cV,\nabla,\{\cV^{p,q}\},h_\bV)$ be a complex variation of Hodge structure of weight $k$. Let
$$\nabla=\overline{\theta}+\partial+\dbar+\theta$$
be the decomposition according to (\ref{align_Griffiths transversality}).
The corresponding higgs bundle of $(\cV,\nabla,h_\bV)$ is $(H={\rm Ker}\dbar,\theta,h_\bV)$ by \cite[\S 8]{Simpson1988}. There is moreover an orthogonal decomposition of holomorphic subbundles
$H=\oplus_{p+q=k}H^{p,q}$ where $H^{p,q}=H\cap\cV^{p,q}$ and 
$$\theta(H^{p,q})\subset H^{p-1,q+1}\otimes \Omega_{X^o}.$$
For the reason of degrees, we have $\overline{\theta}(S(\bV))=0$. By Proposition \ref{prop_VHS_tame}, we acquire that $(H={\rm Ker}\dbar,\theta,h_\bV)$ and $S(\bV)$ satisfy the conditions in Proposition \ref{prop_CVHS_Ssheaf}. Hence $S_X(S(\bV),h_\bV)$ is a coherent sheaf.

The following proposition shows that $S_X(S(\bV),h_\bV)$ coincides with Saito's $S$-sheaf associated to $IC_X(\bV)$.
\begin{prop}\label{prop_coincide_Saito_S}
	Let $\bV:=(\bV,\nabla,\cF^\bullet,h_\bV)$ be an $\bR$-polarized variation of Hodge structure. Denote by $IC_X(\bV)$ the intermediate extension of $\bV$ on $X$ as a pure Hodge module and by $S(IC_X(\bV))$ the Saito's $S$-sheaf associated to $IC_X(\bV)$ (\cite{MSaito1991}). Then
	$$S_X(S(\bV),h_\bV)\simeq S(IC_X(\bV)).$$
\end{prop}
\begin{proof}
	See \cite{Schnell2020} or \cite[Theorem 4.10]{SC2021}.
\end{proof}
Now we are ready to show the Koll\'ar package with respect to a tame harmonic bundle.
\begin{thm}
	Let $f:X\rightarrow Y$ be a proper locally K\"ahler morphism from a complex space $X$ to an irreducible complex space $Y$. Assume that every irreducible component of $X$ is mapped onto $Y$. Let $X^o\subset X_{\rm reg}$ be a dense Zariski open subset and $(H,\theta,h)$ a tame harmonic bundle on $X^o$. Let $E\subset H$ be a holomorphic subbundle with vanishing second fundamental form. Assume that $\overline{\theta}(E)=0$. Let $F$ be a Nakano semi-positive vector bundle on $X$. Then the following statements hold.
	\begin{description}
		\item[Torsion Freeness] $R^qf_\ast (S_X(E,h)\otimes F)$ is torsion free for every $q\geq0$ and vanishes if $q>\dim X-\dim Y$.
		\item[Injectivity] If $L$ is a semi-positive holomorphic line bundle so that $L^{\otimes l}$ admits a nonzero holomorphic global section $s$ for some $l>0$, then the canonical morphism
		$$R^qf_\ast(\times s):R^qf_\ast(S_X(E,h)\otimes F\otimes L^{\otimes k})\to R^qf_\ast(S_X(E,h)\otimes F\otimes L^{\otimes k+l})$$
		is injective for every $q\geq0$ and every $k\geq1$.
		\item[Vanishing] If $Y$ is a projective algebraic variety and $L$ is an ample line bundle on $Y$, then
		$$H^q(Y,R^pf_\ast (S_X(E,h)\otimes F)\otimes L)=0,\quad \forall q>0,p\geq0.$$
		\item[Decomposition] Assume moreover that $X$ is a compact K\"ahler space, then $Rf_\ast (S_X(E,h)\otimes F)$ splits in $D(Y)$, i.e. 
		$$Rf_\ast(S_X(E,h)\otimes F)\simeq \bigoplus_{q} R^qf_\ast(S_X(E,h)\otimes F)[-q]\in D(Y).$$
		As a consequence, the spectral sequence
		$$E^{pq}_2:H^p(Y,R^qf_\ast(S_X(E,h)\otimes F))\Rightarrow H^{p+q}(X,S_X(E,h)\otimes F)$$
		degenerates at the $E_2$ page.
	\end{description}
\end{thm}
\begin{proof}
	By Proposition \ref{prop_CVHS_Ssheaf}, $S_X(E,h)$ is a coherent sheaf. Let $h_F$ be a hermitian metric on $F$ with Nakano semi-positive curvature. By Lemma \ref{lem_Kernel}, we see that
	$$S_X(E,h)\otimes F\simeq S_X(E\otimes F,h\otimes h_F)$$
	is a coherent sheaf on $X$. Note that $(E\otimes F,h\otimes h_F)$ is Nakano semi-positive by Proposition \ref{prop_CVHS_Ssheaf}. It follows from  the abstract Koll\'ar package in \S 4 that the claims of the theorem are valid.
\end{proof}
Taking $S_X(S(\bV),h_\bV)\simeq S(IC_X(\bV))$ for a variation of Hodge structure $\bV$ and $F=\sO_X$, we obtain Koll\'ar's conjecture.

We end this section with remarks on two other packages of Koll\'ar's conjecture. 
\begin{rmk}[Remarks on the Intersection cohomology package]
	In \cite[\S 5.8]{Kollar1986_2}, Koll\'ar also predicts that $S(IC_X(\bV))$ is related to the intersection complex $IC_X(\bV)$ when $\bV=(\cV,\nabla,\cF^\bullet,h)$ is a polarized variation of Hodge structure. This involves the $L^2$-representation of the intersection complex.
	\begin{thm}
		Let $X$ be a compact K\"ahler space of pure dimension $n$ and $X^o\subset X_{\rm reg}$ a dense Zariski open subset. Let $\bV=(\cV,\nabla,\cF^\bullet,h)$ be an $\bR$-polarized variation of Hodge structure of weight $r$ on $X^o$. Then $IH^k(X,\bV)$ admits a pure Hodge structure of weight $k$: 
		$$IH^k(X,\bV)=\bigoplus_{\stackrel{p,q\geq 0}{p+q=k+r}}IH^{p,q}(X,\bV),\quad IH^{p,q}(X,\bV)=\overline{IH^{q,p}(X,\bV)}$$
		 for every $0\leq k\leq 2\dim X$.
		There is moreover a morphism
		$$IC_X(\bV)\to S_X(\bV)$$
		in the derived category of sheaves of $\bC$-vector spaces which induces an isomorphism
		$$IH^{n+r,q}(X,\bV)\to H^q(X,S_X(\bV)),\quad \forall q\geq 0.$$
	\end{thm}
\begin{proof}
	The first statement is a consequence of \cite[Theorem 1.4]{SC2021_CGM}. Roughly speaking, there is a complete K\"ahler metric $ds^2$ on $X^o$ whose $L^2$-de Rham complex $\sD^\bullet_{X,\bV;ds^2,h}$ is quasi-isomorphic to $IC_X(\bV)$. As a consequence, there is a canonical isomorphism
	$$IH^k(X,\bV)\simeq H^{k}_{(2)}(X^o,\bV;ds^2,h),\quad\forall k.$$
	The $(p,q)$-decomposition of forms in $\sD^\bullet_{X,\bV;ds^2,h}$ provides the Hodge structure on $IH^k(X,\bV)$. See \cite[\S 8.3]{SC2021_CGM} for details. The second claim follows from the diagram
	$$\xymatrix{
	\sD^\bullet_{X,\bV;ds^2,h}\ar[r]^-\tau\ar[d]^{\simeq}& \sD^{n,\bullet}_{X,ds^2}(S(\bV),h)\ar[d]^{\simeq\textrm{ (Theorem \ref{thm_main_local1}, Proposition \ref{prop_coincide_Saito_S})}}\\
	IC_X(\bV) & S_X(\bV)
    }$$
    where $\tau$ is taking the projection to the $S(\bV)$-valued $(n,\bullet)$-component.
\end{proof}
\end{rmk}
\begin{rmk}[Remarks on the direct image package]
	In \cite[\S 5.8]{Kollar1986_2}, Koll\'ar also predicts that $$R^qf_\ast S(IC_X(\bV))\simeq S_Y(R^qf_\ast IC_X(\bV)|_{Y^o})$$
	where $Y^o$ is the Zariski open subset of $Y$ so that $R^qf_\ast IC_X(\bV)|_{Y^o}$ is a local system whose fiber at $y\in Y^o$ is canonically isomorphic to $\bH^q(X_y,IC_{X_y}(\bV|_{X_y\cap X^o}))$. This is also a consequence of the $L^2$-representation of $IC_X(\bV)$ (\cite[Theorem 1.4]{SC2021_CGM}). 
\end{rmk}
\bibliographystyle{plain}
\bibliography{CGM_Kollar}

\end{document}